\documentclass[preprint]{imsart}
\RequirePackage{amsthm,amsmath,amsfonts,amssymb,bbm}
\RequirePackage[colorlinks,citecolor=blue,urlcolor=blue]{hyperref}
\RequirePackage{graphicx}

\usepackage{stmaryrd}
\usepackage{comment}
\usepackage{multirow,mathtools,latexsym, enumerate, enumitem, pdftricks,tikz,framed,color,enumitem,enumerate,mleftright,mathrsfs}
\usetikzlibrary{arrows.meta, decorations.markings,math} 
\usepackage{romanbar}

\pdfoutput=1

\startlocaldefs
\newcommand{\rn}[1]{\Romanbar{#1}}

\def\G{\mathsf{G}}
\def\cA{c_A}
\def\m{\mathsf{m}}
\def\n{\mathsf{n}}
\def\E{\mathbb{E}}
\def\Pr{\mathbb{P}}

\def\TV{\mathrm{TV}}
\def\Unif{\operatorname{Unif}}
\def\PiBG{\Pi_{\operatorname{BG}}}
\endlocaldefs

\newtheorem{theorem}{Theorem}[section]

\newtheorem{lemma}[theorem]{Lemma}

\newtheorem{remark}[theorem]{Remark}

\newtheorem{assumption}[theorem]{Assumption}
\newtheorem{example}[theorem]{Example}

\theoremstyle{remark}
\newtheorem{definition}[theorem]{Definition}

\theoremstyle{remark}

\definecolor{darkred}{rgb}{0.9,0.1,0.1}

\begin{document}

\tikzset{->-/.style={decoration={
  markings,
  mark=at position #1 with {\arrow{>}}},postaction={decorate}}}


\begin{frontmatter}

\title{Couplings for Andersen Dynamics}
\runtitle{Couplings for Andersen Dynamics}

\begin{aug}
\Author[A]{\fnms{Nawaf} \snm{Bou-Rabee}\ead[label=e1]{nawaf.bourabee@rutgers.edu}}
\and
\Author[B]{\fnms{Andreas} \snm{Eberle}\ead[label=e2]{eberle@uni-bonn.de}}
\address[A]{Department of Mathematical Sciences \\ Rutgers University Camden \\ 311 N 5th Street \\ Camden, NJ 08102 \\ \printead{e1}}
\address[B]{Institut f\"{u}r Angewandte Matehmatik \\  Universit\"{a}t Bonn \\ Endenicher Allee 60 \\  Bonn, Germany 53115 \\ \printead{e2}}

\runAuthor{N. Bou-Rabee \and A.~Eberle}
\end{aug}

\begin{abstract}
Andersen dynamics is a standard method for molecular simulations, and a precursor of the Hamiltonian Monte Carlo algorithm used in MCMC inference. The stochastic process corresponding to Andersen dynamics is a PDMP (piecewise deterministic Markov process) that iterates between Hamiltonian flows and velocity randomizations of randomly selected particles.  Both from the viewpoint of molecular dynamics and MCMC inference, a basic question is to understand the convergence to equilibrium of this PDMP particularly in high dimension.   Here we present couplings to obtain sharp convergence bounds in the Wasserstein sense that do not require global convexity of the underlying potential energy.
\end{abstract}



\end{frontmatter}


\noindent \today


\section{Introduction} A common task in molecular dynamics is to simulate a molecular system at a specified temperature  \cite{AlTi1987,FrSm2002}.  
The first method suggested for this purpose goes back to Andersen \cite{An1980, ELi2008}. The stochastic process corresponding to Andersen dynamics is a piecewise deterministic Markov process (PDMP)  \cite{davis1984piecewise, Da1993} that combines Hamiltonian trajectories with velocity randomizations of randomly selected particles such that the resulting PDMP leaves the canonical or Boltzmann-Gibbs distribution invariant \cite{Li2008,ELi2008}.  The durations between consecutive velocity randomizations are i.i.d.~exponential random variables with constant mean determined by a collision frequency parameter, and in between these velocity randomizations, the PDMP follows pure Hamiltonian dynamics.  Andersen dynamics is currently implemented in several molecular dynamics software packages including AMBER and GROMACS \cite{Amber19,Gromacs19} and because of its simplicity and reliability continues to be employed in a wide variety of molecular dynamics simulations \cite{bolhuis2002transition,UbAnVo2004,VaWaAl2018,novikov2019ring,ball2019conformational}.

Besides molecular dynamics, Andersen dynamics plays an important conceptual role in Markov chain Monte Carlo (MCMC) inference.  Indeed, Hamiltonian Monte Carlo (HMC) can be viewed as a refinement of Andersen dynamics to include a Metropolis accept/reject step 
\cite{neal1995bayesian}.  Due to the ability of HMC to overcome the diffusive behavior that limits more conventional MCMC methods like Gibbs, random walk Metropolis and the Metropolis adjusted Langevin algorithm, HMC has garnered a great deal of attention in Bayesian statistics \cite{Ne2011,HoGe2014,mangoubi2017rapid,BoEbZi2020,DurmusMoulinesSaksman,livingstone2019,heng2019unbiased}.

Both from the viewpoint of molecular dynamics and MCMC inference, a basic question with Andersen dynamics is to understand the convergence to equilibrium as a function of the collision frequency parameter particularly in high dimension.  If the collision frequency is too small, then on average the integration times of the Hamiltonian trajectories are very long and the PDMP mainly follows Hamiltonian dynamics which by itself is not ergodic in general; whereas if the collision frequency is too high, then the PDMP will exhibit diffusive behavior and it will again take a long time to sufficiently converge. Nevertheless, like other processes that involve Hamiltonian dynamics, one may hope that Andersen dynamics can achieve faster convergence than random walk based methods if the collision frequency is suitably chosen.

 First steps to understand the convergence of Andersen dynamics in terms of the collision frequency have been taken.  Mixing time bounds for Andersen dynamics on a torus were derived in \cite{ELi2008} by showing that Doeblin's condition holds, and subsequently better bounds were obtained in Theorem 6.5 of \cite{Li2007} in the `free-streaming' case where the potential is switched off. 

Here we consider Andersen dynamics for systems with weakly anharmonic potential energies in an unbounded space, and non-convex, twice continuously differentiable potential energies on a high-dimensional torus with weak interactions between particles.    In these settings, we obtain quantitative bounds for the convergence of Andersen dynamics in a Wasserstein sense.  These bounds reveal that if the collision frequency is suitably chosen, then Andersen dynamics can overcome diffusive convergence behavior.  Moreover, these bounds give an optimal dimension dependence.  We use coupling techniques to obtain these bounds. These techniques are based on the framework introduced in \cite{Eb2016A}, and can be viewed as a continuous-time analog on phase space of recently developed couplings for HMC applied to general non-convex models \cite{BoEbZi2020} and high-dimensional mean-field models \cite{BoSc2020}. The coupling used is itself a PDMP, and at least formally, the analysis is based on bounding the action of the generator of the coupling process on distances tailored to each system considered.

We end this introduction by remarking that the tools developed in this paper might be relevant to quantify mixing times of related PDMPs proposed for molecular dynamics and MCMC inference algorithms including zig-zag and bouncy particle samplers \cite{deligiannidis2018randomized,deligiannidis2019exponential,bierkens2019zig,bierkens2019ergodicity,holderrieth2019cores}.



\section{Andersen dynamics and couplings} \label{sec:andersen}

 In this section, we briefly recall Andersen dynamics and its basic properties needed throughout the paper. Then we introduce a new class of couplings for two copies of the dynamics starting at different initial conditions.

\subsection{Andersen dynamics} \label{subsec:andersen}

Andersen dynamics describes a molecular system at constant $\beta = (k_B T)^{-1}$ where $T$ is the temperature and $k_B$ is the Boltzmann constant.  Here we consider a molecular system consisting of $\m$ particles each with $\n$ dimensions.  A state of the molecular system is denoted by $(x,v) \in \mathbb{R}^{2 \m \n}$ where  $x=(x_1, \dots, x_{\m})$ represents the positions of the particles and $v=(v_1, \dots, v_{\m})$ the corresponding velocities. Let $U: \mathbb{R}^{\m \n} \to \mathbb{R}$ denote the potential energy of the molecular system, and for simplicity,  suppose that all particles have unit masses.   Hence, the Hamiltonian of the molecular system is \[
H(x,v) = (1/2) |v|^2 + U(x) \;.
\]

To precisely define Andersen dynamics, let
\begin{eqnarray} \label{hamiltonian_op} 
	\phi_t(x,v) &:=& \left(x_t(x,v), v_t(x,v)\right)\qquad ( t\in [0,\infty ))
\end{eqnarray}
denote the flow of the Hamiltonian dynamics 
\begin{eqnarray}\label{eq:Hamdyn}
	\frac{d}{dt} x_t  \ = \ v_t, \quad \frac{d}{dt} v_t \ = \ - \nabla U(x_t),\quad
	\left(x_0(x,v),v_0(x,v)\right) \ = \ (x,v) \;.
\end{eqnarray}
  For $\mathsf{a} \in \mathbb{R}^{\n}$ and $\mathsf{i} \in \{1, \dots, \m \}$, define the $\mathsf{i}$-th particle velocity substitution  \begin{eqnarray} \label{sub_op} \mathcal{S}(\mathsf{i}, \mathsf{a})(x, v) & := & (x, (v_1, \dots, v_{\mathsf{i}-1}, \mathsf{a}, v_{\mathsf{i}+1} , \dots, v_{\m})) \;.
  \end{eqnarray} 
  As seen below, this map is notationally convenient for describing the velocity randomization of a randomly selected particle in Andersen dynamics.
On the same probability space,
let $(N_t)_{t \ge 0}$ be a homogeneous Poisson process with intensity $\lambda >0$ called the collision frequency in Andersen dynamics, and let $(T_k)_{k \in \mathbb{N}}$ be the corresponding strictly increasing sequence of jump times;  let $(I_k)_{k \in \mathbb{N}}$ and $(\xi_k)_{k \in \mathbb{N}}$ be independent sequences of i.i.d.~random variables $I_k \sim \Unif\{1,...,\m\}$ and $\xi_k \sim \mathcal{N}(0,\beta^{-1})^{\n}$.  The sequence of random variables $(I_k)_{k \in \mathbb{N}}$ represents the indices of the particles whose velocities get randomized to $(\xi_k)_{k \in \mathbb{N}}$ at the jump times $(T_k)_{k \in \mathbb{N}}$ respectively.
  
  \smallskip

  With this notation, the stochastic process $(X_t, V_t)$ corresponding to Andersen dynamics is defined as follows.

\smallskip

\begin{definition}[Andersen Process] \label{D:andersen} Given $t > 0$, $\lambda > 0$ and an initial condition $(x, v) \in \mathbb{R}^{2 \m \n}$, define $T_0=0$, $\delta T_{k} = T_k - T_{k-1}$ for $k \ge 1$, $(X_0,V_0) = (x,v)$ and \begin{eqnarray*}
(X_t, V_t) & := & \phi_{t - T_{N_t}} \circ  \mathcal{S}(I_{N_t}, \xi_{N_t}) \circ \phi_{\delta T_{N_t}} \circ \cdots \circ \mathcal{S}(I_{1}, \xi_{1}) \circ \phi_{\delta T_1}  (X_0, V_0)  \;.
\end{eqnarray*} 
\end{definition}

\smallskip


The process $(X_t, V_t)$ can also be defined piecewise.  In particular,  the process follows Hamiltonian dynamics in between two consecutive jump times, i.e., \[
(X_s, V_s) = \phi_{s-T_{k-1}} (X_{T_{k-1}},V_{T_{k-1}})\quad\text{for }s \in [T_{k-1}, T_k)\text{ and }k \ge 1.
\]
Moreover, at a jump time, $s=T_k$, the velocity of the $I_k$-th particle instantaneously changes to $\xi_k$, i.e.,
\[
(X_{T_k}, V_{T_k}) = \mathcal{S}(I_k, \xi_k) (X_{T_k-}, V_{T_k-})
\] where $(X_{T_k-}, V_{T_k-}) = \phi_{T_k-T_{k-1}} (X_{T_{k-1}},V_{T_{k-1}})$.  

The dynamics generates a Piecewise Deterministic Markov Process (PDMP) on the state space $\mathbb{R}^{2 \m \n}$.  The law of a PDMP is determined by one or several vector fields which govern its deterministic motion, a measurable function which gives the law of the random times between jumps, and a jump measure which gives the transition probability of its jumps \cite{davis1984piecewise, Da1993}.  
In the case of Andersen dynamics, these are given by: \begin{itemize}
\item the vector field \[
\mathfrak{X}(x,v) =  ( v, - \nabla U(x) ) \;, \quad (x,v) \in \mathbb{R}^{2 \m \n} \;,
\] generating the deterministic Hamiltonian flow;
\item the (constant) jump rate given by the collision frequency $\lambda$; and,
\item the jump measure \[
Q((x,v), (d x^{\prime} \; d v^{\prime})) =  \frac{1}{\m} \sum_{i=1}^{\m} \delta_{x}(d x^{\prime})\varphi_{\beta}(v^{\prime}_i) d v_i^{\prime} \prod_{j \ne i} \delta_{v_j}(d v^{\prime}_j)  \;,
\] 
where $\varphi_{\beta}(v_i^{\prime}) = (2 \pi / \beta )^{-\n/2} \exp(-(\beta/2) |v_i^{\prime}|^2)$.
\end{itemize}

By \cite[Theorem 5.5]{davis1984piecewise}, the corresponding PDMP $(X_t, V_t)$ with given initial condition $(x,v)$  solves the local martingale problem for the extended generator $(\mathcal{G}, \mathcal{D}(\mathcal{G}))$ defined by
\begin{equation}
\mathcal{G} f = \mathcal{L} f + \mathcal{A} f  \;, \quad f \in \mathcal{D}(\mathcal{G}) \;.    
\end{equation}
Here $\mathcal{D}(\mathcal{G})$ is the set of all continuously differentiable functions $f: \mathbb{R}^{2 \m \n} \to \mathbb{R}$,  \begin{equation}
 \mathcal{L} f(x,v) = \mathfrak{X}(x,v) \cdot \nabla f(x,v) = v \cdot \nabla_x f(x,v) - \nabla U(x) \cdot \nabla_v f(x,v)   \label{eq:L}
\end{equation} 
is the {\em Liouville operator} associated to the Hamiltonian dynamics, and 
\begin{equation} \label{eq:A}
 \mathcal{A} f(x,v) = \lambda  \E \left\{ f( \mathcal{S}(I,\xi) (x, v) ) - f(x,v) \right\}  
\end{equation}
is the {\em Andersen collision operator} where the expectation in \eqref{eq:A} is over the independent random variables $I \sim \Unif\{1,\dots,\m\}$ and $\xi  \sim \mathcal{N}(0,\beta^{-1})^{\n}$.


A key property of Andersen dynamics is that it leaves invariant the Boltzmann-Gibbs probability distribution 
\begin{equation} \label{eq:boltzmann_gibbs}
\PiBG(dx\,dv) \propto \exp(-\beta H(x,v))\, dx\, dv \;.
\end{equation}   
Indeed, since the Hamiltonian flow preserves both the Hamiltonian function $H$ and phase space volume (as a consequence of symplecticity), the Hamiltonian flow preserves $\PiBG$.    Moreover, since the position component is held fixed and the $i$th velocity component is drawn from the $v_i$-marginal of $\PiBG$, the velocity randomizations also preserve $\PiBG$.
This argument can be easily turned into a proof that $\PiBG$ is infinitesimally invariant in the sense that \[
\int_{\mathbb{R}^{2 \m \n}} \mathcal{G} f(z)  \PiBG(dz) = 0
\] for any compactly supported $C^1$ function $f: \mathbb{R}^{2 \m \n} \to \mathbb{R}$. To conclude that $\PiBG$ is an invariant measure (not just infinitesimally invariant) requires additional assumptions on $U$, e.g., it is sufficient to show that an appropriate Foster-Lyapunov drift condition holds; see \S3.5 of \cite{BoSa2017} for details.

\begin{remark}
In the case of one particle $\m=1$ with $\beta=1$, Andersen dynamics becomes exact randomized Hamiltonian Monte Carlo (xrHMC) which is geometrically ergodic under mild conditions on the potential energy $U$ \cite{BoSa2017}.  
\end{remark}

\begin{remark}
Andersen dynamics is related to second-order Langevin dynamics, but there are differences.  First, note that Andersen dynamics does not incorporate explicit dissipation or diffusion.  Second, although the velocity randomizations help ensure that Andersen dynamics is ergodic with respect to the Boltzmann-Gibbs probability distribution, they have the disadvantage of introducing jump discontinuities along the velocity of trajectories.  In contrast, the velocity of trajectories for second-order Langevin dynamics is continuous.  
\end{remark}


\subsection{Couplings for Andersen Dynamics} \label{sec:coupling}

A key tool in our analysis is a Markovian coupling $Y_t = ((X_t, V_t), (\tilde X_t, \tilde V_t))$ of two realizations of Andersen dynamics starting from different initial conditions.   To precisely define this coupling, introduce the following Hamiltonian flow on $\mathbb{R}^{4 \m \n}$
\begin{eqnarray} \label{coup_hamilt_op} 
	\phi^C_t((x,v),(\tilde x, \tilde v)) &:=& \left(\phi_t( x,  v), \phi_t(\tilde x,\tilde v)) \right) \qquad ( t\in [0,\infty ))
\end{eqnarray}
where $\phi_{t}$ is the Hamiltonian flow from \eqref{hamiltonian_op}.
Let $\gamma \ge 0$ be a parameter of the coupling whose precise value will be specified in an appropriate way in subsequent sections.  For $\mathsf{a} \in \mathbb{R}^{\n}$, $\mathsf{i} \in \{1, \dots, \m \}$, and $u \in (0,1)$, introduce   \begin{eqnarray} \label{coup_sub_op}
\mathcal{S}^C(i,\mathsf{a},u)((x,v),(\tilde x, \tilde v)) & := &  (\mathcal{S}(i,\mathsf{a})(x,v), \mathcal{S}(i, \tilde{\mathsf{a}}) (\tilde x, \tilde v))
\end{eqnarray} where $\mathcal{S}(\cdot, \cdot)$ is the mapping in \eqref{sub_op} and $\tilde{\mathsf{a}} = \Phi(\mathsf{a}, z_i, u)$ with $z_i = x_i - \tilde x_i$.  Here we have introduced the function $\Phi: \mathbb{R}^{\n} \times \mathbb{R}^{\n} \times (0,1) \to \mathbb{R}^{\n}$ defined by
 \begin{equation} \label{eq:tildexi}
\Phi(\mathsf{a},\mathsf{b},u) \ := \ \begin{dcases}
\mathsf{a} + \gamma \mathsf{b} & \text{if $u< \dfrac{\varphi_{\beta}(\mathsf{a} + \gamma \mathsf{b})}{\varphi_{\beta}(\mathsf{a})}$} \;, \\
\mathsf{a} - 2 (e_{\mathsf{b}} \cdot \mathsf{a}) e_{\mathsf{b}} & \text{else}  \;,
\end{dcases}
\end{equation}
where $e_{\mathsf{b}} = \mathsf{b}/ |\mathsf{b}|$ for $\mathsf{b} \ne 0$ and $e_0 = 0$. 

\medskip
A simple calculation gives the following estimates that will be used below.
\begin{lemma}\label{lem:rejp} 
Let $\mathsf{b} \in \mathbb{R}^{\n}$, and let
$\tilde \xi = \Phi(\xi,\mathsf{b}, \mathcal{U}  )$ where $\xi \sim \mathcal{N}(0,\beta^{-1})^{\n}$ and $\mathcal{U} \sim \Unif(0,1)$.  Then
\begin{align}
\label{ieq:rejpA}\mathbb P[\xi - \tilde \xi \neq -\gamma \mathsf{b} ] \ &\le \ \sqrt{\beta} \gamma  |\mathsf{b} |/\sqrt{2\pi} ,\qquad\qquad\text{and}\\
\label{ieq:rejpB}\mathbb E[|\xi |^2;\, \xi - \tilde \xi \neq -\gamma \mathsf{b} ] \ &\le \ (\n+1) \gamma  |\mathsf{b} |/\sqrt{2\pi \beta} \, .
\end{align}
\end{lemma}
The lemma is a refinement of Lemma 3.7 in 
\cite{BoEbZi2020}. A self-contained proof is provided in Section \ref{proof:rejp}.

\medskip

Let $(N_t)_{t \ge 0}$ be a Poisson counting process with intensity $\lambda$ and let $(T_k)_{k \in \mathbb{N}}$ be the corresponding strictly increasing sequence of jump times; let $(I_k)_{k \in \mathbb{N}}$,  $(\xi_k)_{k \in \mathbb{N}}$, and $(\mathcal{U}_k)_{k \in \mathbb{N}}$ be independent sequences of i.i.d.~random variables $I_k \sim \Unif\{1,...,\m\}$, $\xi_k \sim \mathcal{N}(0,\beta^{-1})^{\n}$, and $\mathcal{U}_k \sim \Unif(0,1)$, all defined on a joint probability space.
\smallskip
  
With this notation, we define the following coupling for Andersen dynamics.  

\smallskip

\begin{definition}[Coupling for Andersen Dynamics]
\label{D:coupling}
Given $t>0$, $\lambda>0$, $\gamma \ge 0$, and an initial condition $y \in \mathbb{R}^{4 \m \n}$, define $T_0=0$, $\delta T_{k} = T_k - T_{k-1}$ for $k\ge 1$, $Y_0 = y$, and \begin{equation*}
Y_t \, := \, \phi^C_{t - T_{N_t}} \circ   \mathcal{S}^C(I_{N_t}, \xi_{N_t}, \mathcal{U}_{N_t}) \circ   \phi^C_{\delta T_{N_t}} \circ  \cdots \circ  \mathcal{S}^C(I_{1}, \xi_{1}, \mathcal{U}_{1}) \circ  \phi^C_{\delta T_{1}}  (Y_0)  \;.
\end{equation*} 
\end{definition}

\smallskip

The process $Y_t$ can also be defined piecewise.  In particular, the components of the coupling follow Hamiltonian dynamics in between two consecutive jump times, \[
Y_s = \phi^C_{s-T_{k-1}} ( Y_{T_{k-1}} )  \;,~~ \text{for $s \in [T_{k-1},T_k)$} \;. 
\]Moreover, at a jump time, $s=T_k$, the velocities of the $I_k$-th particles in the first and second components of the coupling process instantaneously change to $V_{T_k}^{I_k} = \xi_k$ and $\tilde V_{T_k}^{I_k} = \Phi(\xi_k,X_{T_k}^{I_k} - \tilde X_{T_k}^{I_k},\mathcal{U}_k)$ respectively, i.e., \[
Y_{T_k} = \mathcal{S}^C(I_k, \xi_k, \mathcal{U}_k)\ ( Y_{T_k-})
\] where $Y_{T_k-} =  \phi^C_{T_k-T_{k-1}} ( Y_{T_{k-1}} ) $.  We stress that $X_{T_k} = X_{T_k-}$, $\tilde X_{T_k} = \tilde X_{T_k-}$, and $V_{T_k}^j = V_{T_k-}^j$,  $\tilde V_{T_k}^j = \tilde V_{T_k-}^j$ for $j \in \{1, \dots, \m \} \setminus \{ I_k \}$.

\medskip

This coupling of Andersen dynamics is inspired by recently introduced couplings for Hamiltonian Monte Carlo \cite{BoEbZi2020} and second-order Langevin dynamics \cite{eberle2019couplings}.
 It is motivated by the observation that the free-streaming Hamiltonian dynamics is contractive for small time durations if the difference in the initial velocities is chosen negatively proportional to the difference in the initial positions \cite[Figure 1]{BoEbZi2020}.  In particular, the velocity randomization at a jump time $T_k$ is defined such that the difference in the velocities of the $I_k$-th particles satisfies $V_{T_k}^{I_k} - \tilde{V}_{T_k}^{I_k} = \xi_{I_k} - \tilde \xi_{I_k} = - \gamma ( X_{T_k}^{I_k} - \tilde{X}_{T_k}^{I_k})$ with maximal possible probability, and otherwise, a reflection coupling is applied, as illustrated in Figure~\ref{fig:coupling_of_velocities}.
 
\medskip

\begin{figure}[ht!]
\begin{center}
\includegraphics[width=0.49\textwidth]{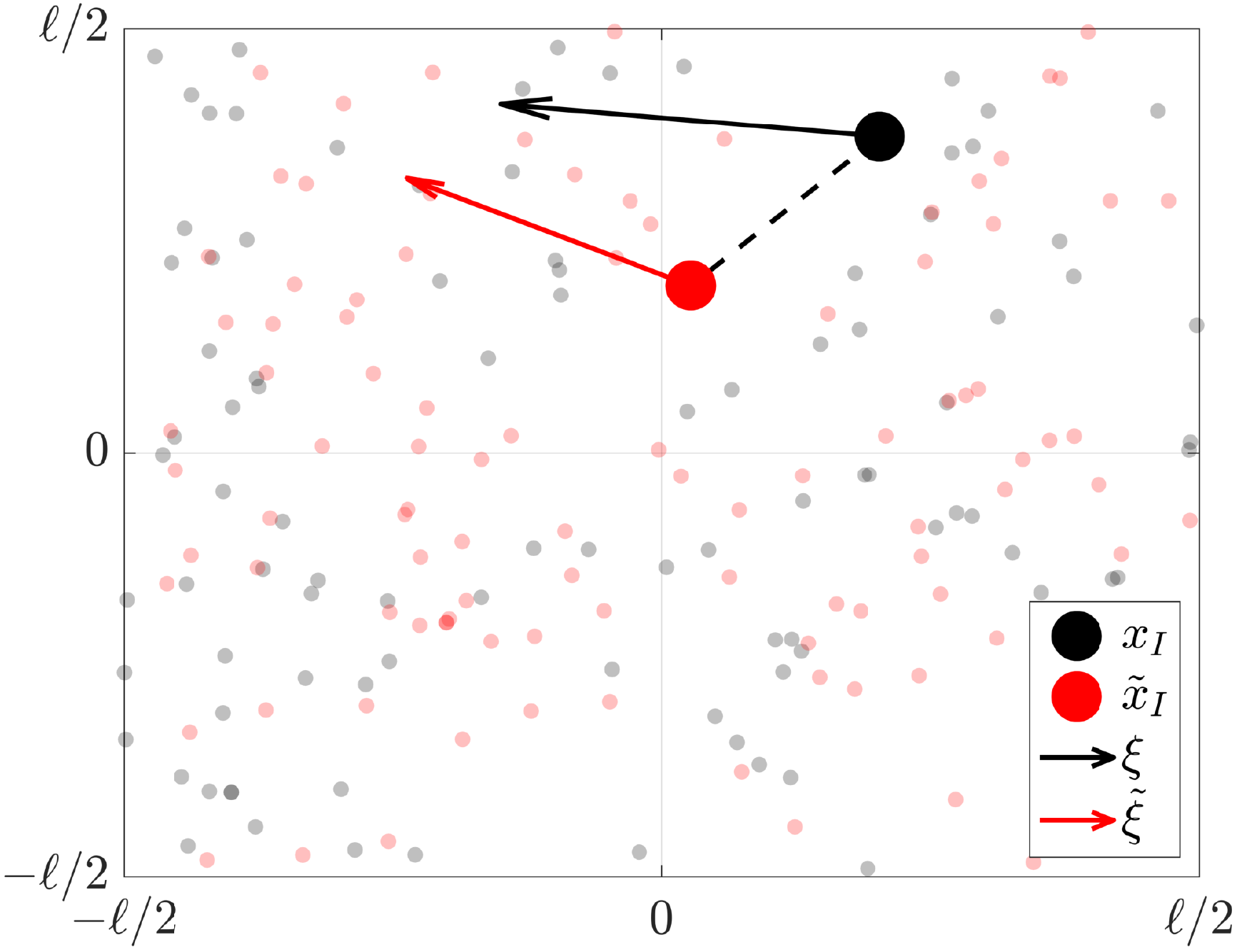}  \hfill
\includegraphics[width=0.49\textwidth]{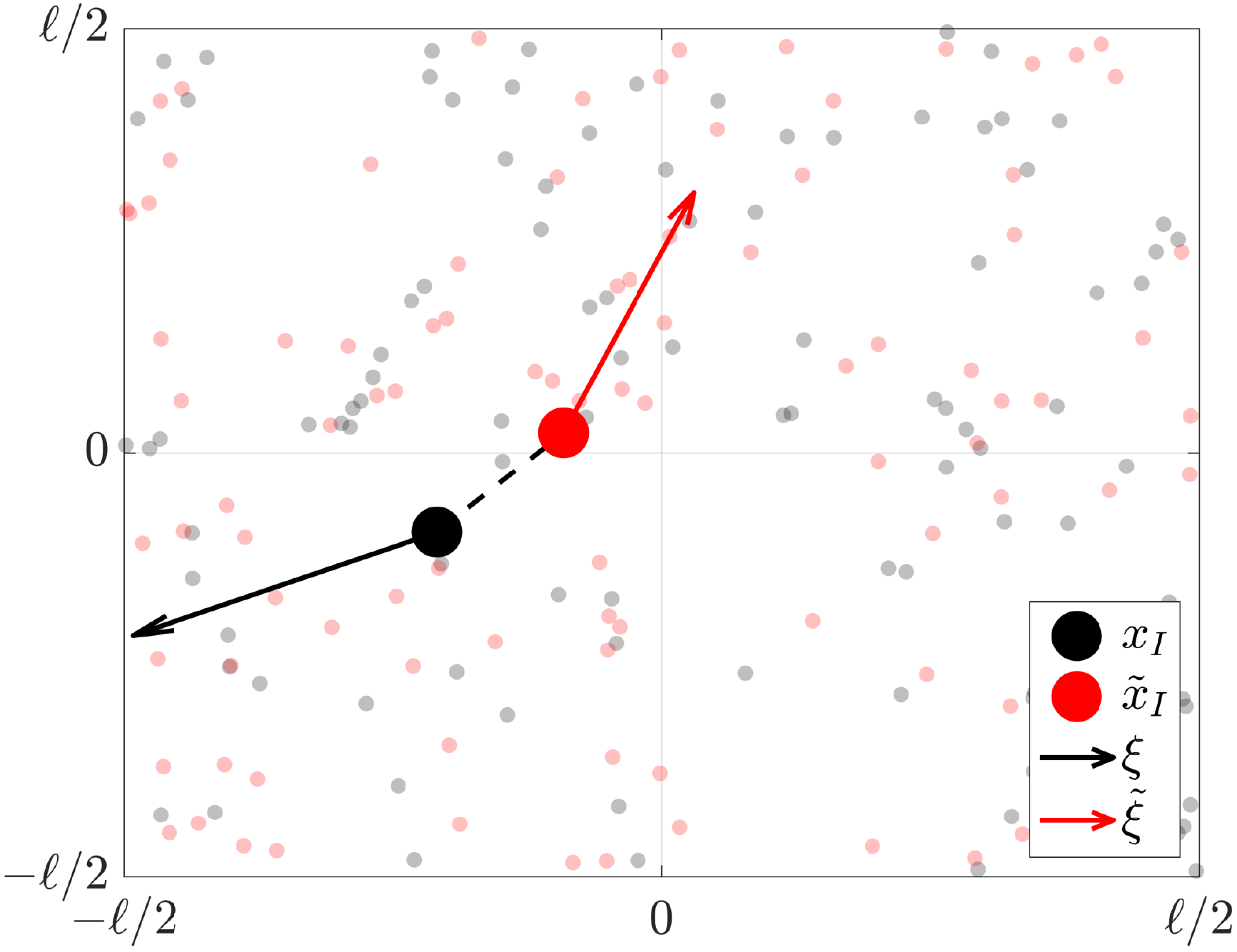}  \\
\end{center}
\begin{minipage}{\textwidth}
      \begin{minipage}[b]{0.49\textwidth}
      \centering
(a) $\tilde{\xi} = \xi + \gamma (x_I - \tilde x_I)$ 
\end{minipage}
\begin{minipage}[b]{0.49\textwidth}
\centering
(b) $\tilde{\xi} = \xi - 2 (e_{z} \cdot \xi) e_{z} $ 
\end{minipage}
\end{minipage}
\caption{ \small  Illustrations of a velocity randomization step where the difference in velocities of the $I$-th particles (enlarged dots), $\xi-\tilde \xi$, is:  (a) negatively proportional to the difference in their positions $z = x_I - \tilde x_I$; and (b) reflected about the hyperplane passing through the origin, orthogonal to $e_{z} = z / |z|$. 
 }
 \label{fig:coupling_of_velocities}
\end{figure}

The coupling process $Y_t$ is itself a PDMP on the state space $\mathbb{R}^{4 \m \n}$ with the following characteristics: \begin{itemize}
\item the vector field \begin{equation*} 
\mathfrak{X}^C((x,v),(\tilde x, \tilde v)) = ( v , - \nabla U(x) , \tilde v , - \nabla U(\tilde x) ) \;;
\end{equation*}
\item the (constant) jump rate given by the collision frequency $\lambda$; and,
\item the jump measure \begin{equation*} 
\begin{aligned}
& Q^C ( y ,  d y ^{\prime} ) =  \\
& \qquad \frac{1}{\m} \sum_{i=1}^{\m} \delta_{x}(d x^{\prime}) \delta_{\tilde x}(d \tilde x^{\prime}) 
Q^C_i( (v_i, \tilde v_i), (d v_i^{\prime} \; d \tilde v_i^{\prime}) ) \prod_{j \ne i} \delta_{v_j}(d v_j^{\prime}) \delta_{\tilde v_j}(d \tilde v_j^{\prime})    \;, \\
& \text{where}~ 
Q^C_i((v_i, \tilde v_i), (d v_i^{\prime} \; d \tilde v_i^{\prime}) ) = \left( \varphi_{\beta}(v_i^{\prime}) \wedge \varphi_{\beta}(v_i^{\prime} + \gamma z_i ) \right) \delta_{v_i^{\prime}+\gamma z_i}(d \tilde v^{\prime}_i)  d v_i^{\prime} \\
& \qquad\qquad + \left( \varphi_{\beta}(v_i^{\prime}) - \varphi_{\beta}(v_i^{\prime} + \gamma z_i) \right)^+  \delta_{v_i^{\prime} -2 (e_{z_i} \cdot  v_i^{\prime} ) e_{z_i}}(d \tilde v^{\prime}_i)  d v_i^{\prime}
\end{aligned}
\end{equation*} 
and where $e_{z_i} = z_i/|z_i|$ for $z_i \ne 0$ and $e_0=0$. 
\end{itemize} 

Since the coupling process is again a PDMP, the results in \cite{davis1984piecewise} show that it solves a local martingale 
problem for an extended generator 
\begin{equation} \label{eq:GC}
\mathcal{G}^C_{\gamma}  =   \mathcal{L}^C   + \mathcal{A}^C_{\gamma}  
\end{equation} 
which is the sum of the Liouville operator $\mathcal{L}^C$ for the Hamiltonian vector field $\mathfrak{X}^C$ and
a velocity randomization operator $\mathcal{A}^C_{\gamma}$, and whose domain
$\mathcal{D}(\mathcal{G}^C_{\gamma})$ consists of continuously differentiable functions on $\mathbb{R}^{4 \m \n}$.
For a function $F: \mathbb{R}^{4 \m \n} \to \mathbb{R}$ that is differentiable at $y$, the Liouville operator $\mathcal{L}^C$ is given by
\begin{equation} \label{eq:LC}
\mathcal{L}^C   F (y) = \mathfrak{X}^C(y) \cdot \nabla F(y) \;.
\end{equation}
The action of the coupled velocity randomization operator $\mathcal{A}^C_{\gamma}$ on a function $F: \mathbb{R}^{4 \m \n} \to \mathbb{R}$ is defined as
  \begin{eqnarray} \label{eq:AC}
 \mathcal{A}^C_{\gamma} F(y)  &=& \lambda \E \left\{  F(\mathcal{S}^C(I,\xi,\mathcal{U}) y) - F(y) \right\}\\
\nonumber & =& \lambda \int_{\mathbb{R}^{4 \m \n}} \left( F(y^{\prime}) -  F(y) \right)  Q^C(y, d y^{\prime}) \;.
\end{eqnarray}
where 
the expectation is taken over the independent random variables $I \sim \Unif\{1,\dots,\m\}$, $\xi \sim \mathcal{N}(0,\beta^{-1})^{\n}$ and $\mathcal{U} \sim \Unif(0,1)$.

It can be easily verified that the process $(Y_t)_{t\ge 0}$ is indeed a coupling of two copies of Andersen dynamics. Indeed, by uniqueness of the local martingale problem for the Andersen process, it is sufficient to check that $\mathcal{G}^C_{\gamma} F(y)$ reduces to $\mathcal{G} F(x,v)$ or $\mathcal{G} F(\tilde x,\tilde v)$  for 
functions independent of the first or second component of $y=((x,v), (\tilde x, \tilde v))$, respectively. For such functions, it immediately follows that $\mathcal{L}^C F = \mathcal{L} F$, and using a similar calculation to the one performed in Section 2.3.2 of \cite{BoEbZi2020},  $\mathcal{A}^C_{\gamma} F = \mathcal{A} F$;  hence, $\mathcal{G}^C_{\gamma} F = \mathcal{G} F$.

\begin{remark}[Synchronous coupling] \label{rmk:sync}
When $\gamma=0$ in Definition~\ref{D:coupling} the velocities of the $I_k$-th particles are synchronously randomized, i.e., $V_{T_k}^{I_k} = \tilde{V}_{T_k}^{I_k} = \xi_k$.  For the corresponding generators, we write $\mathcal{A}^C_{sync} = \mathcal{A}^C_0$ and $\mathcal{G}^C_{sync} = \mathcal{G}^C_{0}$.  By itself, a synchronous coupling is insufficient to obtain contractivity for non-strongly-convex potentials.  
\end{remark}
 

 \begin{figure}[t!]
\begin{center}
\includegraphics[width=0.3\textwidth]{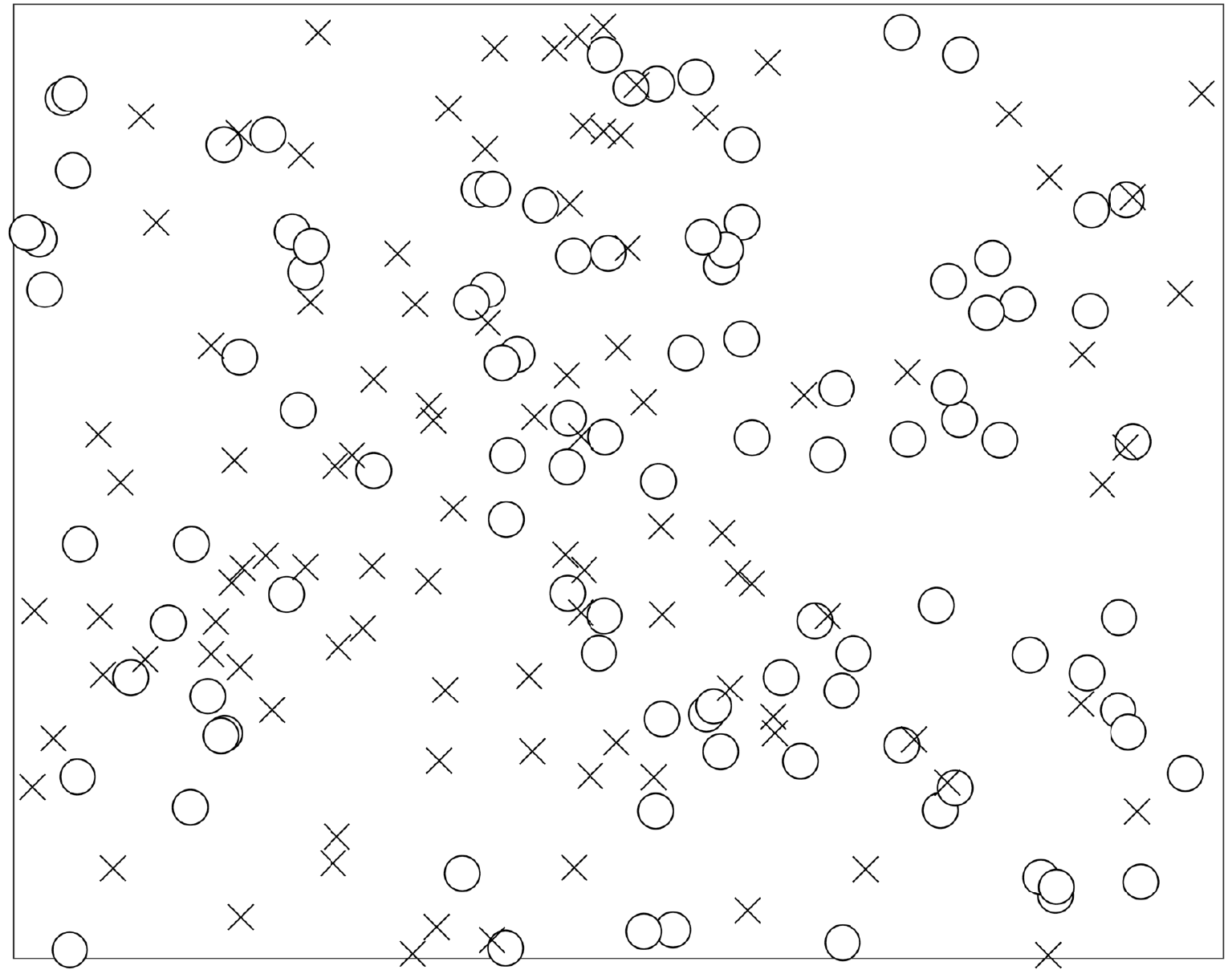}  \hfill
\includegraphics[width=0.3\textwidth]{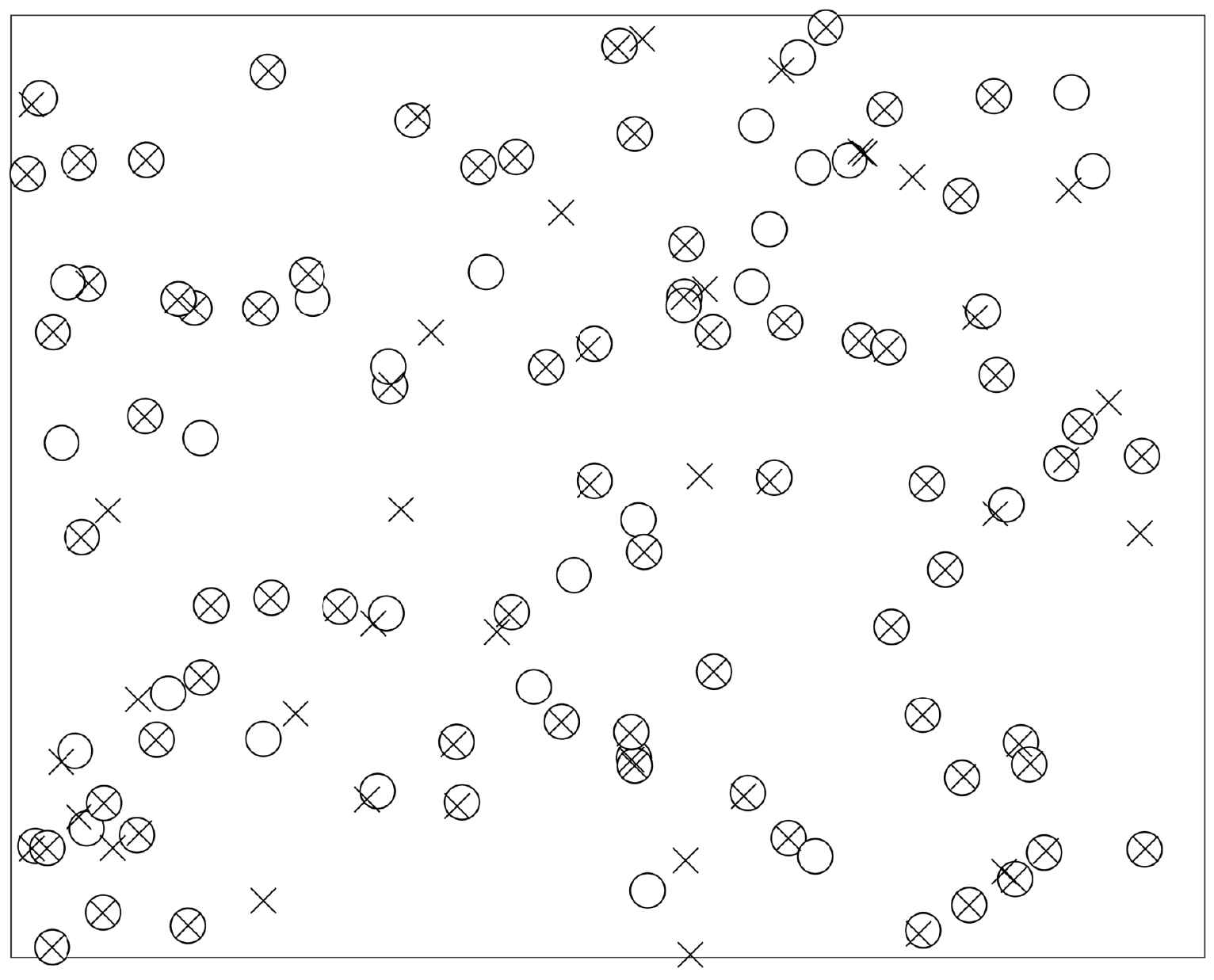} 
\hfill
\includegraphics[width=0.3\textwidth]{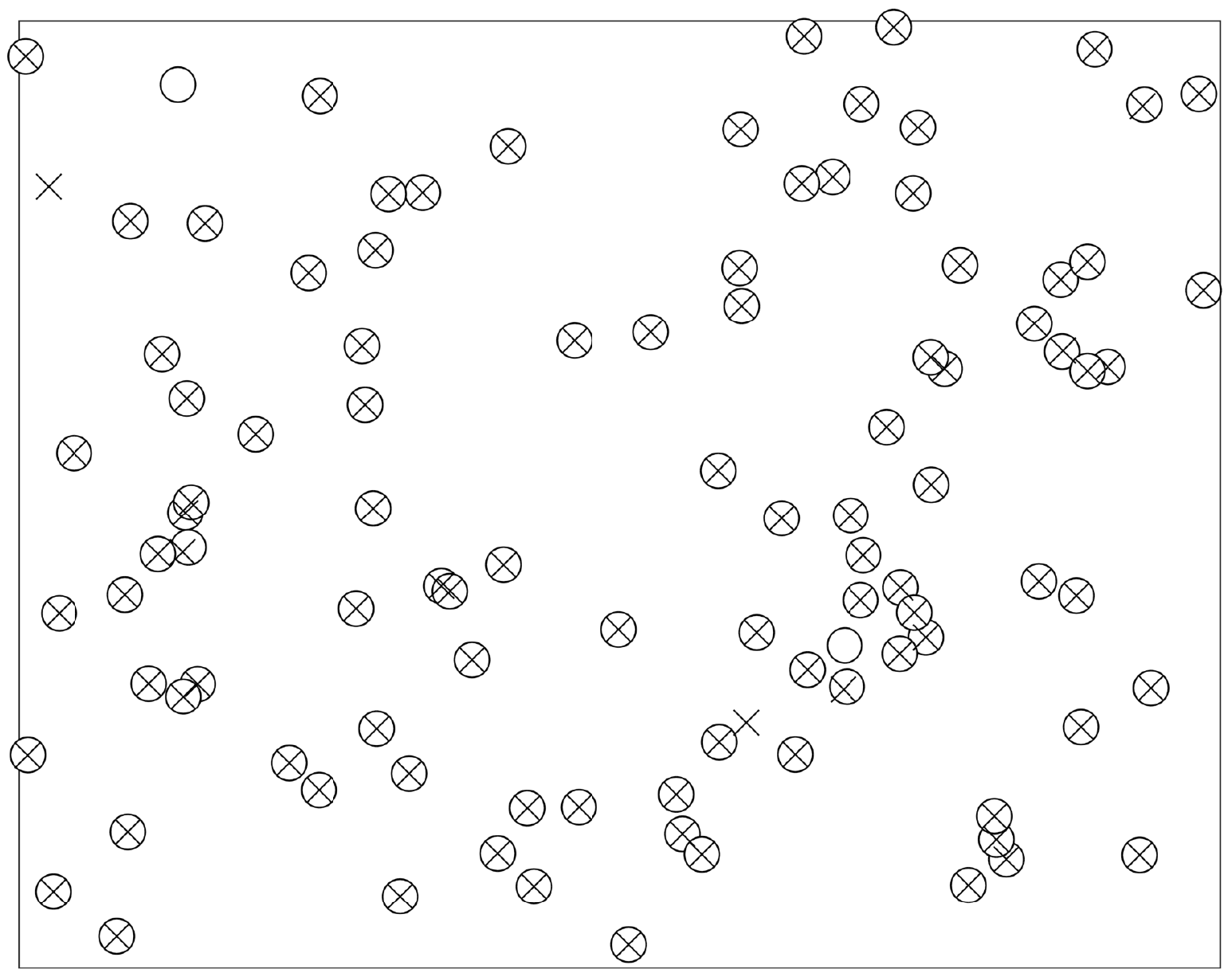} 
\\
\end{center}
\begin{minipage}{\textwidth}
      \begin{minipage}[a]{0.32\textwidth}
      \centering
(a) $t=1$
\end{minipage}
\begin{minipage}[b]{0.33\textwidth}
\centering
(b) $t=4$
\end{minipage}
\begin{minipage}[c]{0.32\textwidth}
\centering
(c) $t=8$
\end{minipage}
\end{minipage}
\caption{ \small  Snapshots at the indicated times along a realization of the coupling for Andersen dynamics on $(\mathbb{T}^{2 \m}_{\ell} \times \mathbb{R}^{2 \m})^2$ in the free-streaming case with $\m=100$, $\ell=1$, and $\beta=1$.  
 }
 \label{fig:coupling_snapshots}
\end{figure}


\begin{figure}[ht!]
\begin{center}
\includegraphics[width=0.49\textwidth]{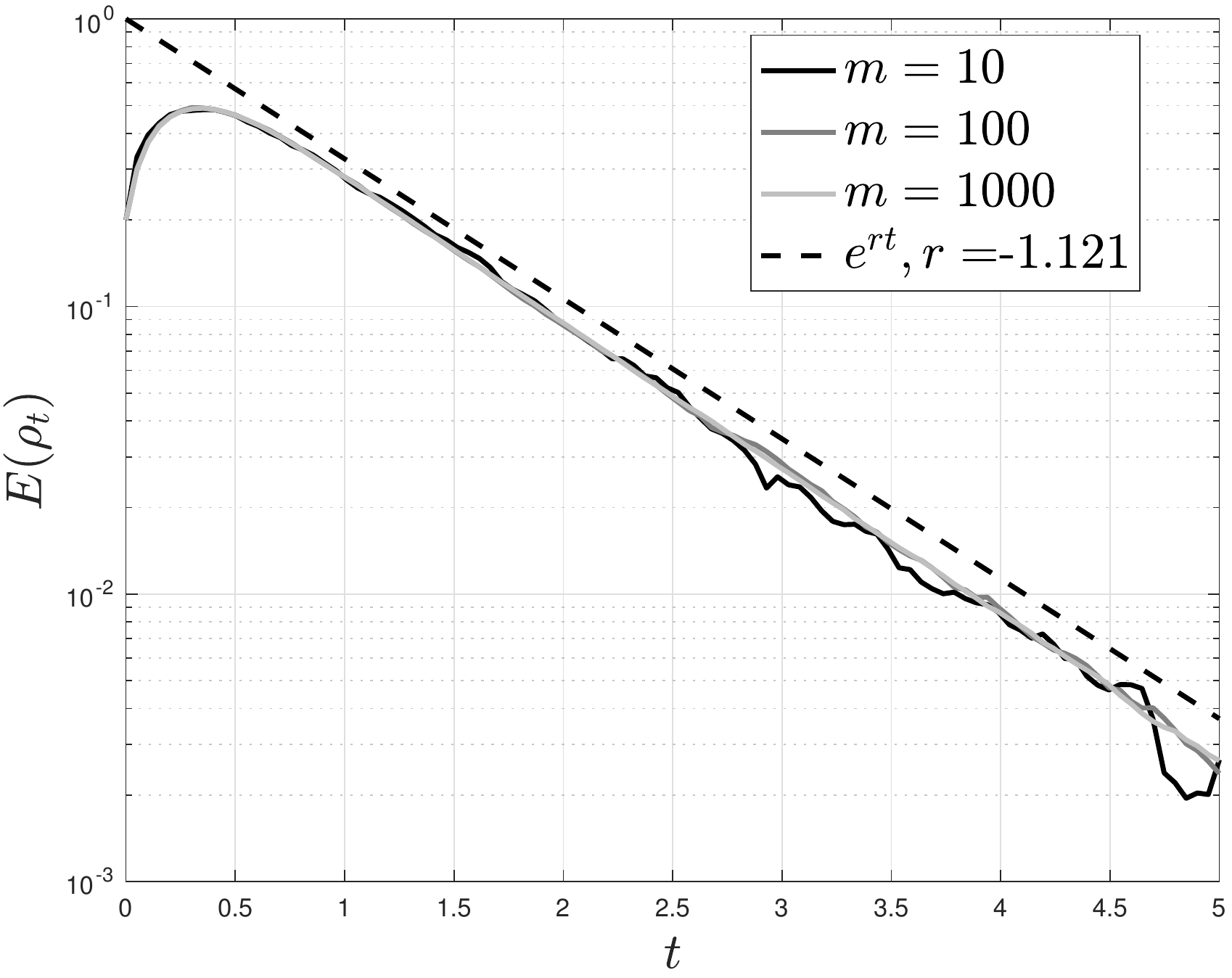}  \hfill
\includegraphics[width=0.49\textwidth]{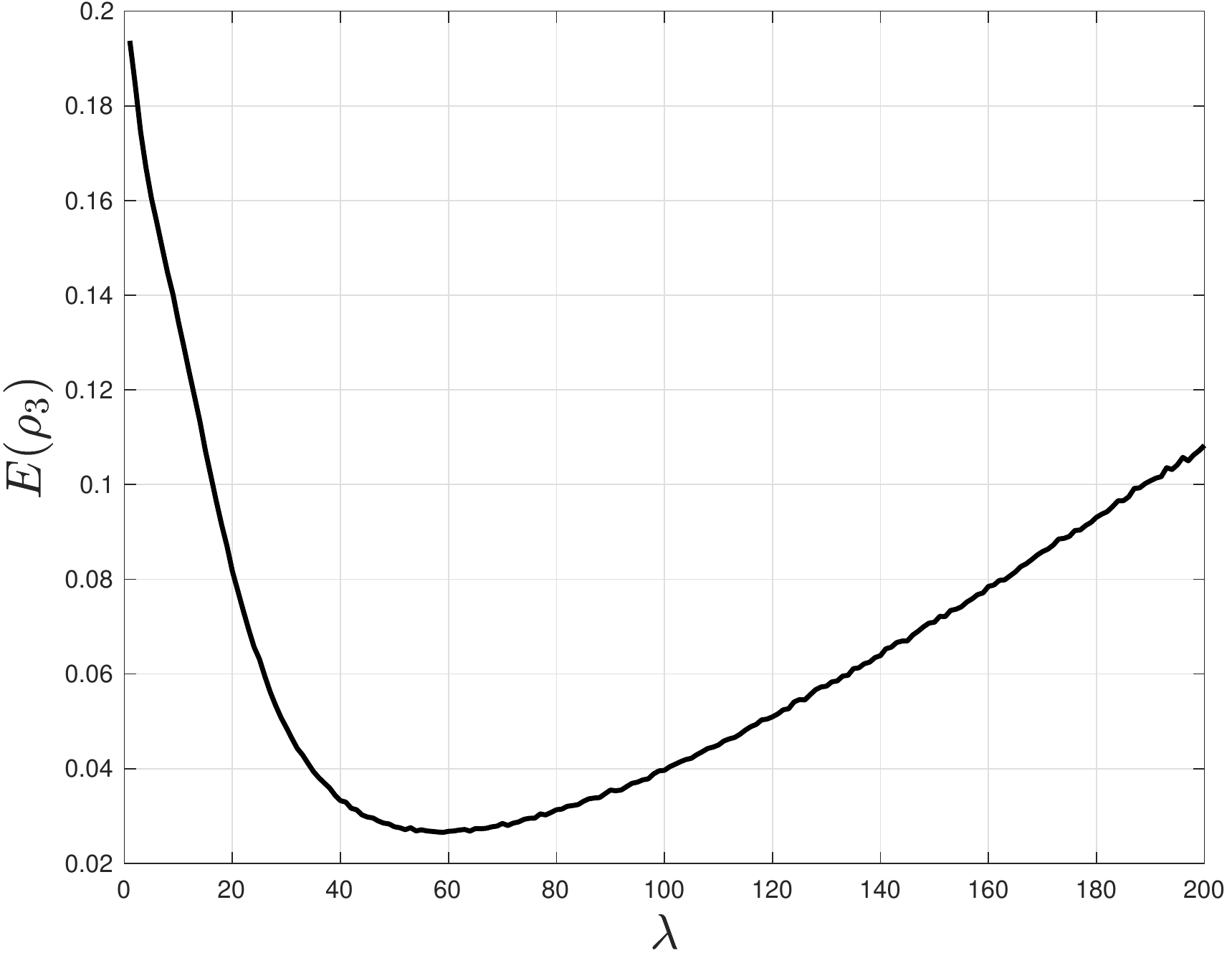}  \\
\end{center}
\begin{minipage}{\textwidth}
      \begin{minipage}[b]{0.49\textwidth}
      \centering
(a) $\lambda/\m=6$
\end{minipage}
\begin{minipage}[b]{0.49\textwidth}
\centering
(b) $\m=10$
\end{minipage}
\end{minipage}
\caption{ \small Evolution of a Monte Carlo estimate of $\E(\rho_t)$ using $10^5$ realizations of the coupling for Andersen dynamics on $(\mathbb{T}^{\m}_{\ell} \times \mathbb{R}^{\m})^2$ where $\ell=1$, $U\equiv0$, $\beta=1$, $\rho_t = (1/\m) \sum_{i=1}^{\m} \sqrt{ \zeta_i(Z_t,W_t)^2+(W_t^i)^2 } $, and  $\gamma=1/(\ell/2+\m/\lambda)$.  This choice of $\gamma$ is motivated by Thm.~\ref{TOR}. In (a), $\lambda/\m$ is fixed while $\m$ is increased from 10 to 1000; note that the observed convergence rate is dimension-free and consistent with Thm.~\ref{TOR} which implies contractivity with respect to an equivalent metric.   In (b), $\E(\rho_3)$ is plotted as a function of $\lambda$; note that $\lambda$ in (a) is approximately the minimizer of $\E(\rho_3)$. 
 }
 \label{fig:coupling}
\end{figure}

\subsection{Andersen dynamics on a torus} 

Molecular dynamics simulations routinely employ periodic boundary conditions \cite{AlTi1987,FrSm2002,LeRoSt2010,Gromacs19,Amber19}.  In particular, the configuration space of the molecular system is typically a flat torus $\mathbb{T}_{\ell}^{\m \n}$. Here  $\mathbb{T}_{\ell} =\mathbb{R}/ ( \ell \mathbb{Z})$ denotes the circle with circumference $\ell>0$. 
The canonical projection from the covering space $\mathbb{R}^{\m \n}$ to the torus $\mathbb{T}_{\ell}^{\m \n}$ is denoted by $\pi$, and $\tau_z(x)\in \mathbb{T}_{\ell}^{\m \n}$ denotes the translation of
a point $x\in\mathbb{T}_{\ell}^{\m \n} $ by a tangent vector $z\in\mathbb{R}^{\m \n} $. 

\medskip

Let $U \in C^2(\mathbb{T}^{\m \n}_{\ell})$ satisfy $U(x)\ge 0$ for all $x\in\mathbb{T}_{\ell}^{\m \n}$. Andersen dynamics on the torus ${T}_{\ell}^{\m \n}$ with potential $U$ is the PDMP with state space $ {T}_{\ell}^{\m \n}\times \mathbb{R}^{\m \n}$ defined by Definition \ref{D:andersen}, where $\phi_t$ is now the flow of Hamiltonian dynamics \eqref{eq:Hamdyn} on the torus, and $\mathcal S$ is again defined by \eqref{sub_op} as above. The process
can also be obtained by projection from Andersen dynamics on Euclidean space. Indeed,
let $\hat U$ denote the periodic function in $C^2(\mathbb{R}^{\m \n})$ defined 
by $\hat U(x)=U(\pi (x))$ for all $x$. Then the
Andersen process $(X_t,V_t)$ on the torus with inital 
condition $(x_0,v_0)\in {T}_{\ell}^{\m \n}\times \mathbb{R}^{\m \n}$ is given by 
$X_t=\pi (\hat X_t)$ and $\hat V_t=V_t$, where
$(\hat X_t,\hat V_t)$ is the Andersen process 
on $\mathbb R^{\m\n} \times \mathbb{R}^{\m \n}$ with initial condition 
$(\hat x_0,v_0)$ for an arbitrary $\hat x_0\in \pi^{-1}(x_0)$.

\begin{figure}[t]
\centering
\begin{minipage}{\textwidth}
\begin{minipage}{0.45\textwidth}
\begin{center}
\begin{tikzpicture}[scale=1.0]
    \draw[draw=blue,->-=.5,thick] (2,-1) to (0,-1);
        \draw[draw=blue,->-=.5,thick] (0,-1) to (-2,-1);
      \draw[->] (-2.5,0) -- (2.5,0) node[right] {\large
$\zeta$};
      \draw[->] (0,-1.5) -- (0,1.5) node[above] {\large
 $w$};
      \draw[-] (-2,-0.1) -- (-2,0.1) node[below=0.15cm] {\normalsize
$-\ell/2$};
      \draw[-] (2,-0.1) -- (2,0.1) node[below=0.15cm] {\normalsize
$\ell/2$};
        \filldraw[color=blue,fill=blue] (2.0,-1) circle (0.07);
        \filldraw[color=red,fill=white] (-2.0,-1) circle (0.07); 
    \end{tikzpicture}
    \end{center}
\end{minipage}
\begin{minipage}{0.45\textwidth}
\begin{center}
\begin{tikzpicture}[scale=1.0]
      \draw[->] (0,0) -- (5.5,0) node[right] {\large
$t$};
      \draw[->] (0,-1.5) -- (0,1.5) node[above] {\large
$\zeta_t$};
      \draw[-] (-0.1,1) -- (0.1,1) node[left=0.2cm] {\normalsize
$\ell/2$};
      \draw[-] (-0.1,-1) -- (0.1,-1) node[left=0.2cm] {\normalsize
$-\ell/2$};
            \filldraw[color=blue,fill=blue] (0,1) circle (0.07);
\foreach \a in {0,...,4}{
            \draw[scale=1.0,domain=(\a+0.001):((\a+1)),variable=\x,blue,thick] plot[samples=100] ({\x},{1-2*\x-2*floor((1-2*\x+1)/2)});
            \draw[dashed,blue] ({(\a+1)},-1) -- ({(\a+1)},1);
            \filldraw[color=blue,fill=blue] ({(\a+1)},1) circle (0.07);
            \filldraw[color=red,fill=white] ({(\a+1)},-1) circle (0.07);
         }        
    \end{tikzpicture}
    \end{center}
\end{minipage}
\end{minipage}
\begin{minipage}{\textwidth}
\begin{minipage}[b]{0.48\textwidth}
      \centering
(a)  
\end{minipage}
\begin{minipage}[b]{0.48\textwidth}
      \centering
(b) 
\end{minipage}
\end{minipage}
\caption{\small  Plots of $\zeta_t=\zeta(z_t,w_t)$ initially at $(z_0,w_0)=(\ell/2,-1)$ with constant $w_t$ in (a) phase space  and (b) as a function of time. Note that  $t \mapsto \zeta_{t}$ is a c\`{a}dl\`{a}g trajectory. }
  \label{fig:zeta}
\end{figure}
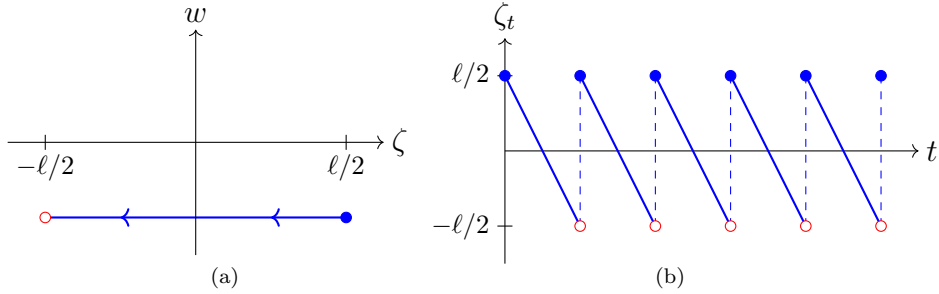

\subsection{Coupling for Andersen dynamics on a torus} 

We now introduce a coupling for two copies of
the Andersen process on the torus. The coupling
is a piecewise deterministic Markov process 
$((X_t,V_t),(\tilde X_t,\tilde V_t))$ with state space 
$ \left(\mathbb{T}_{\ell}^{\m \n}\times \mathbb R^{\m \n}\right)^2$. Although in spirit, the construction is similar
to the construction of a coupling for Andersen dynamics on euclidean space in Section \ref{sec:coupling}, some technical
difficulties arise in the torus case. Therefore, we require
a slightly different setup.

In order to
construct the coupling process, we consider another PDMP 
$Y_t=(X_t,V_t,Z_t,W_t)$ with state space
$\mathbb{T}_{\ell}^{\m \n}\times \mathbb R^{\m \n}\times \mathbb R^{\m \n}\times \mathbb R^{\m \n}$. The coupling is obtained from this process by setting
\begin{equation}\label{eq:couplingprojection}
((X_t,V_t),(\tilde X_t,\tilde V_t)):=\pi^C(Y_t),
\end{equation}
where $\pi^C:\mathbb{T}_{\ell}^{\m \n}\times \mathbb R^{3\m \n}\to (\mathbb{T}_{\ell}^{\m \n}\times \mathbb R^{\m \n})^2$
is the projection map defined by
\begin{equation}\label{eq:piC}
 \pi^C(x,v,z,w)=\left( (x,v),(\tau_{-z}(x),v-w)\right) .  
\end{equation}
Thus $W_t=V_t-\tilde V_t$ and $\tilde X_t=\tau_{-Z_t}(X_t)$,
i.e., $W_t$ and $Z_t$ correspond to the differences between the coupling components.

Let $\phi_t^C=(x_t,v_t,z_t,w_t)$
denote the flow on $\mathbb{T}_{\ell}^{\m\n} \times \mathbb{R}^{3\m\n}$ of the ODE
\begin{equation}
\label{odes:tor}
\begin{aligned}
	\frac{d}{dt} x_t & \ = \ v_t \;, \qquad &  \frac{d}{dt} z_t & \ = \ w_t \;, \\ 
 \frac{d}{dt} v_t & \ = \ - \nabla U(x_t) \;, \qquad & \frac{d}{dt} w_t & \ = \  \nabla U(\tau_{-z_t}(x_t)) - \nabla U(x_t)   \;.
	\end{aligned}
\end{equation}
For $(z,w)\in  \mathbb{R}^{2\m\n}$, we also define $\zeta (z,w) \in [-\ell /2,\ell /2]^{\m\n}$ by
\begin{equation} \label{zeta_i}
 \zeta_{i,j}(z,w) \ = \ 
 \begin{cases}
 z_{i,j} - \lfloor (z_{i,j} + \ell/2) / \ell \rfloor \ell & \text{if $z_{i,j}\not\in\ell/2+\ell\mathbb Z$} \; ,\\
 
 \ell/2 & \text{if $w_{i,j}<0$ and $z_{i,j}\in\ell/2+\ell\mathbb Z$} \;, \\
  -\ell/2 & \text{if $w_{i,j}\ge 0$ and $z_{i,j}\in\ell/2+\ell\mathbb Z$} \;.
 \end{cases} 
\end{equation}
One should think of $\zeta =\zeta (z,w)$ as a minimal difference vector between 
the corresponding components $x$ and $\tilde x$ on the torus. In particular,
$\zeta_{i,j}\equiv z_{i,j}\text{ mod }\ell$ for all $i,j$, and thus
\begin{equation}
    \label{eq:xtxzeta}
    x=\tau_z(\tilde x)=\tau_\zeta (\tilde x) \, .
\end{equation}
The motivation
for the special definition of $\zeta_{i,j}$ for $z_{i,j}\in \ell /2+\ell\mathbb Z$ is that it ensures that if $(x_t,v_t,z_t,w_t)$
is a solution of \eqref{odes:tor} then $t\mapsto \zeta_{i,j} (z_t,w_t)$ is c\`{a}dl\`{a}g (right continuous with left limits)
for all $t$ such that $w_{t}^{i,j} \neq 0$, see Figure \ref{fig:zeta}.
This will imply that the coupling distance introduced further 
below is also a c\`{a}dl\`{a}g function of $t$, see Lemma
\ref{lem:cadlag}.

\medskip

The process $(Y_t)$ is now defined by Definition \ref{D:coupling} above where $\phi_t^C$ is the flow of \eqref{odes:tor}, and
\begin{eqnarray}
\label{coup_sub_op_tor}
\mathcal{S}^C(i,\mathsf{a},u)(x,v,z,w) & = &  (\mathcal{S}(i,\mathsf{a})(x,v), \mathcal{S}(i, \mathsf{a}-\tilde{\mathsf{a}}) (z, w)) \\
\nonumber \text{with}&&\tilde{\mathsf{a}} = \Phi(\mathsf{a}, \zeta_i(z,w), u).
\end{eqnarray} 
Again, $(Y_t)$
is a piecewise deterministic Markov process with generator given by
\eqref{eq:GC}, \eqref{eq:LC} and \eqref{eq:AC},
where now the vector field generating the deterministic flow is $\mathfrak{X}^C(x,v,z,w)=(v,-\nabla U(x),w,\nabla U(\tau_{-z}(x))-\nabla U(x))$, $\mathcal S^C$ is defined by 
\eqref{coup_sub_op_tor}, and the jump measure $Q^C$ is adapted correspondingly.



\section{Main Results}

We now apply the couplings introduced above to derive contraction results and bounds on Wasserstein distances to the invariant measure for Andersen dynamics. We first consider a strongly convex potential energy function on $\mathbb{R}^{\m \n}$. In this case, relatively precise bounds can be derived by 
synchronous coupling. Then we consider Andersen dynamics on a high dimensional torus, which is a common setup in molecular dynamics. In that
case, synchronous coupling can not be applied since the potential energy function is not convex. In general, phase transitions can cause slow mixing as the dimension goes to infinity. Using the couplings 
introduced above, we are able to show that rapid mixing still holds 
for weak interactions between the particles.

\subsection{Andersen dynamics for weakly anharmonic molecular systems}

Here we consider potentials $U(x)$ that satisfy the following assumption.
\begin{assumption} \label{hypo:U}
The potential energy is weakly anharmonic, i.e., \begin{equation} U(x) \ = \ \frac{1}{2}  x^T \mathcal{C}^{-1} x + G(x) \;, \quad \text{for all $x \in \mathbb{R}^{\m \n}$} \;, \label{eq:U_wah}
\end{equation} where $\mathcal{C}$ is an $\m \n \times \m \n$ symmetric positive definite matrix; and
the perturbation $G(x)$ is a convex, continuously differentiable and $L_G$-gradient Lipschitz function, i.e.,  there exists $L_G \ge 0$ such that \begin{eqnarray}    	\label{eq:wah_lip}    	    	|\nabla G(x)-\nabla G(\tilde x)|  & \leq &   L_G |x-\tilde x| \;, \quad \text{$x, \tilde x \in \mathbb{R}^{\m \n}$} \;.    \end{eqnarray}  
\end{assumption}
Any $K$-strongly convex, continuously differentiable and gradient Lipschitz function $U(x)$ can be put in the form of \eqref{eq:U_wah} with $\mathcal{C} = K^{-1} \mathbf{1}_{\m \n}$ where $\mathbf{1}_{\m \n}$ is the $\m \n \times \m \n$ identity matrix and $G(x) = U(x) - K |x|^2 / 2$.  Moreover, it follows from this assumption that $U(x)$ is itself strongly convex \begin{equation}
    \label{eq:wah_sc}
(\nabla U(x) - \nabla U(\tilde x)) \cdot (x - \tilde x) \ge \sigma_{max}^{-2} |x - \tilde x|^2 \quad \text{for all $x, \tilde x \in \mathbb{R}^{\m \n}$} \;,
\end{equation}
where $\sigma_{max}^2$ is the largest eigenvalue of $\mathcal{C}$.  Here we used the convexity of $G(x)$ which implies that $(\nabla G(x) - \nabla G(\tilde x)) \cdot (x - \tilde x) \ge 0$.  
 The contraction result given below uses a synchronous coupling of velocities to exploit the convexity of the perturbation $G(x)$; see Remark~\ref{rmk:sync}. Let  
$$H_0(x,v):=(1/2) (|v|^2 + x^T \mathcal{C}^{-1} x)$$ 
be the unperturbed Hamiltonian. The Hamiltonian of the weakly anharmonic system is $H(x,v) = H_0(x,v) + G(x)$.  In terms of $H_0$, define the metric $\rho: \mathbb{R}^{4 \m \n} \to \mathbb{R}^+$ by
\begin{align} 
\rho(y)^2 \ &:= \ H_0(z(y),w(y)) + \frac{\lambda}{4 \m} \, z(y) \cdot w(y) + \frac{\lambda^2}{8 \m^2} \, |z(y)|^2 \label{eq:rho} \\
\ &  = \ \begin{pmatrix} z(y) & w(y) \end{pmatrix} \G \begin{pmatrix} z(y) \\ w(y) \end{pmatrix} \;, ~~ \G \ := \ \begin{bmatrix} \frac{\lambda^2}{8 \m^2} \mathbf{1}_{\m \n} + \frac{1}{2} \mathcal{C}^{-1} & \frac{\lambda}{8 \m} \mathbf{1}_{\m \n} \\
\frac{\lambda}{8 \m} \mathbf{1}_{\m \n} & \frac{1}{2} \mathbf{1}_{\m \n} \end{bmatrix} \;, 
\label{eq:rhoG}
\end{align}
where for $y=(x,v,\tilde x,\tilde v)$, we set $z(y) = x - \tilde x$ and $w(y) = v - \tilde v$.
Note that $\rho(y)$ only depends on $\mathcal{C}$ and the intensity of the velocity randomizations per particle $\lambda/\m$.  Moreover, by completing the square in \eqref{eq:rho}, it is easy to show that $\rho(y)^2$ is positive definite. 

In the sequel, we will sometimes write the $y$ dependence in $z$, $w$, $\rho$, etc. and sometimes suppress it in the notation, depending on what is more convenient.  Let $(p_t)_{t \ge 0}$ denote the transition semigroup of Andersen dynamics, and for all probability measures $\mu, \nu$ on $\mathbb{R}^{2 \m \n}$ let $\mathcal{W}_2(\mu, \nu)$ denote the standard 2-Wasserstein distance.

\begin{theorem}\label{WAH}
Suppose that Assumption~\ref{hypo:U} holds and $\lambda>0$ satisfies
\begin{equation} \label{eq:lambcond_wah}
{\lambda}/{\m} \ \ge \ 4 L_G \sigma_{max} \;.
\end{equation}
Then 
\begin{equation}\label{WAH1}
	\mathcal{G}^C_{sync} \, \rho^2  \leq - c \, \rho^2  \, ,\quad\text{where}\quad
 c :=  \frac{\lambda}{\m} \min\left( \frac{1}{8},  \frac{8}{5} \frac{\m^2}{\sigma_{max}^2 \lambda^2}  \right) \, .
\end{equation}
Thus, the process $t\mapsto e^{ c t} \rho (Y_t)^2$ is a nonnegative supermartingale, and \begin{equation} \label{eq:W2_wah}
    \mathcal{W}_2(\mu p_t, \nu p_t) \ \le \ \varkappa(\G)^{1/2} e^{-c t/2} \mathcal{W}_2(\mu, \nu)
\end{equation}
where $\varkappa(\G)$ is the condition number of the matrix $\G$.  
\end{theorem}

 A proof of this theorem is provided in Section~\ref{sec:proofs}.   In the unperturbed case $G \equiv 0$, a similar result is proven for exact randomized HMC in Proposition 4 of \cite{deligiannidis2018randomized}.  Related results have been proven for HMC in \cite{mangoubi2017rapid,BoEbZi2020,chen2019optimal} and second order Langevin dynamics in \cite{dalalyan2018sampling,cheng2018underdamped}, though an important difference in Theorem~\ref{WAH} is that condition~\eqref{eq:lambcond_wah} and the rate in \eqref{WAH1} do not deteriorate in the limit that the condition number of $\mathcal{C}$ becomes large for fixed $\sigma_{max}$.   
 

\begin{example}[Strongly Convex Potential]   \label{ex:SC}
For a $K$-strongly convex, continously differentiable and gradient Lipschitz function $U(x)$, Theorem~\ref{WAH} gives a rate of $c = (\lambda/\m) \min(1/8, (8/5) K \m^2 / \lambda^2)$ provided that $\lambda$ satisfies $\lambda/\m \ge 4 L_G/\sqrt{K}$ where $L_G$ is a Lipschitz constant for the gradient of $G(x) = U(x) - K |x|^2 /2 $.
\end{example}

The next example can be viewed as the potential energy corresponding to a truncation of an infinite-dimensional Gaussian measure \cite{BePiSaSt2011,BlCaSa2014,BoEb2019}.  This model problem illustrates the importance of duration randomization when the underlying Hamiltonian dynamics is highly oscillatory.

\begin{example} [Neal's Example] \label{ex:neal}
Let $\m=1$ and $U(x) = 2^{-1} \sum_{i=1}^{\n} i^{2} x_i^2$; hence, $\sigma_{max} = 1$. The corresponding Hamiltonian dynamics is highly oscillatory when the dimension $\n$ is large \cite{petzold_jay_yen_1997}.  
Noting again that condition~\eqref{eq:lambcond_wah} always holds when $L_G =0$, Theorem~\ref{WAH} gives a dimension-free rate of $c = (\lambda/8) \min ( 1, 16 / (5 \lambda^2) )$ which is maximized at $\lambda^{\star} = 4 \sqrt{5}/5$ where $c^{\star} =\sqrt{5} / 10$.
\end{example}

More generally, when $L_G=0$, the rate from Theorem~\ref{WAH} is maximized at $\lambda^{\star}/\m = 4 \sqrt{5}/(5 \sigma_{max})$ where $c^{\star} = \sqrt{5}/(10 \sigma_{max})$.  This conclusion remains true when $L_G$ is small; specifically, when $ L_G \le \sqrt{5}/(5 \sigma_{max}^2)$.  However, when $L_G$ is larger than that, i.e., $L_G > \sqrt{5}/(5 \sigma_{max}^2)$, the rate is maximized at $\lambda^{\star} / \m = 4 L_G \sigma_{max}$ where  $c^{\star} = 1/(10 L_G \sigma_{max}^3)$.

\begin{remark}[Duration Randomization]  Due to possible periodicity of the Hamiltonian flow, contraction bounds for HMC in the strongly convex case typically require that the duration parameter is short enough \cite{mangoubi2017rapid,BoEbZi2020,chen2019optimal}.  On the other hand, since duration randomized Hamiltonian flows avoid periodicities almost surely, contraction bounds for exact randomized HMC allow longer mean durations as illustrated in Example~\ref{ex:neal} \cite{Ma1989, Ne2011, BoSa2017, deligiannidis2018randomized}.
\end{remark}

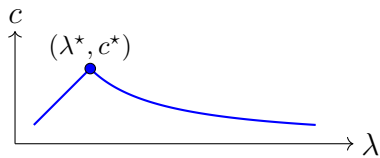
\begin{figure}[t]
\centering
\begin{tikzpicture}[scale=1.0]
      \draw[->] (0,0) -- (4.5,0) node[right] {\large $\lambda$};
      \draw[->] (0,0) -- (0,1.5) node[above] {\large $c$};
      \draw[scale=1.0,domain=0.25:4,variable=\x,blue,thick] plot[samples=200] ({\x},{min(\x,1/\x)});
      \filldraw[color=black,fill=blue] (1,1) circle (0.07) node [above,black, scale=1.0]  {\normalsize $(\lambda^{\star}, c^{\star})$};
    \end{tikzpicture}
\caption{\small  The rates appearing in Example~\ref{ex:neal} is maximized at $(\lambda^{\star}, c^{\star})$.  }
  \label{fig:optimal_lamb}
\end{figure}



\subsection{Contractivity of Andersen dynamics with weak interactions on a high-dimensional torus}

In this part, in order to avoid overloading the notation, we assume $n=1$. However, we stress that the results below can be extended
without essential changes to the case $n\neq 1$. In the following,
we assume that the potential energy of the molecular system $U: \mathbb{T}^{\m} \to \mathbb{R}$ is nonnegative and twice continuously differentiable.

\begin{assumption} \label{hypo::U}
The potential energy $U \in C^2(\mathbb{T}^{\m}_{\ell})$ satisfies  $U(x) \ge 0$ for all $x \in \mathbb{T}^{\m}_{\ell}$.
\end{assumption}

Assumption~\ref{hypo::U} implies that the following constants are finite: \begin{align} \label{eq:LJ}
    L = \sup_{\substack{1 \le i \le \m \\ x \in \mathbb{T}_{\ell}^{\m}}} \left| \frac{\partial^2 U}{\partial x_i^2}(x) \right| \;, \quad  J = \sup_{\substack{1 \le i < j \le \m \\ x \in \mathbb{T}_{\ell}^{\m}}} \left| \frac{\partial^2 U}{\partial x_i \partial x_j}(x) \right| \;.
\end{align}
Fix $i \in \{1, \dots, \m \}$ and define $\nabla_i U(x) \, := \, \frac{\partial U}{\partial x_i}(x)$. Then for all $x, \tilde x \in \mathbb{T}^{\m}_{\ell}$,  \begin{align}
& \left| \nabla_i U(x) - \nabla_i U(\tilde x) \right| \  = \ \left| \int_0^1 \frac{d}{ds} \nabla_i U (\tau_{ s \zeta}(\tilde x))\, ds \right| \nonumber \\
& \qquad \le \left| \int_0^1  \frac{\partial^2 U}{\partial x_i^2} (\tau_{ s \zeta}(\tilde x))\, \zeta_i \, ds \right| +  \sum_{k \ne i} \left| \int_0^1  \frac{\partial^2 U}{\partial x_i \partial x_k} (\tau_{ s \zeta}(\tilde x))\, \zeta_k\,  ds \right| \nonumber \\
& \qquad \le L | \zeta_i | + J \sum_{k \ne i} |\zeta_k | \label{ieq:LJ}
\end{align} 
where $\zeta\in \mathbb R^{\m}$ is an arbitrary tangent vector such that $x=\tau_\zeta (\tilde x)$.

\medskip

The choice of an adequate metric in order to prove contraction properties on the torus is quite tricky. It combines ideas from several previous works
including in particular the results on contractive couplings for
Langevin dynamics and HMC in \cite{Eb2016A}, \cite{eberle2019couplings} and \cite{BoEbZi2020}, as well as the derivation of dimension-free contraction rates
for mean-field models with weak interactions in
\cite{Eb2016A} and \cite{BoSc2020}. Besides combining
these approaches, they have to be adapted
to the special setup on the torus.

To each pair $\left((x,v),(\tilde x,\tilde v)\right)\in \left( \mathbb{T}_{\ell}^{\m} \times \mathbb{R}^{ \m} \right)^2$, we assign $y=(x,v,z,w)\in\mathbb{T}_{\ell}^{\m} \times \mathbb{R}^{3 \m} $ such that $w=v-\tilde v$ and $x=\tau_z(\tilde x)$, i.e.,
$\left((x,v),(\tilde x,\tilde v)\right)=\pi^C(y)$. We define
$ \zeta =\zeta (z,w)$ by \eqref{zeta_i}, and we set $$ q(z,w)=\zeta (z,w)+\gamma^{-1}w.$$
Although the choice of $y$ is not unique, the definition of $\zeta$
and $q$ does not depend on this choice. Since $\zeta$ is in $[-\ell/2,\ell/2]^\m$, it is a tangent vector of a minimal geodesic 
from $\tilde x$ to $x$. Let $\alpha>0$ and let $i \in \{1, \dots, \m \}$. With a slight abuse of notation, we now define a weighted $\ell_2$-distance between the $i$-th components of the coupling by  \begin{equation} \label{eq:ri}
\begin{aligned}
r_i \left((x,v),(\tilde x,\tilde v)\right)\, =\, r_i(y) \, = \, \sqrt{ |\zeta_i(z,w)|^2+ \alpha^{-2} |q_i(z,w)|^2 }  \; .
\end{aligned}
\end{equation} 
Given an initial condition $y \in \mathbb{T}^{\m}_{\ell} \times \mathbb{R}^{3 \m}$, and for any $t \ge 0$, let $y_t=(x_t,v_t,w_t,z_t)$ be the solution to \eqref{odes:tor} with $y_0 = y$, and for any $i \in \{1, \dots, \m \}$, let $r^i_t := r^i(y_t)$ and $\zeta^i_t := \zeta_i(z_t,w_t)$.  As illustrated in Figure~\ref{fig:r}, and as presented in the lemma below, the definition in \eqref{eq:ri} is motivated by the property that $t \mapsto r_t^i := r_i(y_t)$ is a c\`{a}dl\`{a}g trajectory.

\begin{lemma} \label{lem:cadlag}
The function $t \mapsto r_t^i$ is c\`{a}dl\`{a}g and lower semi-continuous, i.e.,   $r^i_t = \lim_{s \downarrow t} r^i_s  \le \lim_{s \uparrow t} r^i_s$ for any $t \ge 0$.  Moreover, it is continuous at points $t$ such that $|\zeta^i_t| < \ell/2$ or $w^i_t = 0$.
\end{lemma}
A proof of Lemma~\ref{lem:cadlag} is provided in Section~\ref{sec::proofs:regularity}.  

\medskip

Let $a,\mathcal{R}>0$,  and define a function $f: \mathbb{R}_{\ge0} \to \mathbb{R}_{\ge0}$ as \begin{equation} \label{eq:f}
f(r) \ := \  \int_0^r e^{-a t} \ \mathbbm{1}_{\{ t \le \mathcal{R}\}} \ dt \ = \ \frac 1a\left( 1-e^{-a\, r\wedge \mathcal{R}}\right) \;.
\end{equation}
Note that $f$ is nondecreasing, concave, bounded, and both constant and maximal when $r \ge \mathcal{R}$. Moreover, for all $s, r \ge 0$, \begin{equation} \label{ieq:f}
    f(s) - f(r) \le f_-'(r) \min(s-r, a^{-1}) \;,
\end{equation}
where $f_-'(r)$ is the left derivative of $f(r)$.
To measure the distance between the components of the coupling process, we use the following distance function \begin{equation} \label{eq::rho}
\rho \left((x,v),(\tilde x,\tilde v)\right)\, =\, \rho(y) \ = \ \sum_{i=1}^{\m} f ( r_i(y) )  \;.
\end{equation}
This definition is motivated by \cite{Eb2016A} and \cite{BoSc2020} where similar distance functions have been introduced to obtain dimension-free contraction rates for (resp.) Langevin dynamics and HMC applied to models with weak interactions. For probability measures $\mu ,\nu$ on $\mathbb{T}_{\ell}^{\m} \times \mathbb{R}^{\m}$, we define \begin{equation}\label{eq:Wrho}
\mathcal{W}_{\rho}(\mu , \nu) \ := \ \inf_{\substack{(X, V) \sim \mu \\ (\tilde X, \tilde V) \sim \nu}} \E \left[ \rho((X,V),(\tilde X, \tilde V)) \right] 
\end{equation} 
where the infimum is over all couplings of $\mu$ and $\nu$.   

\medskip

We can now state our {main contraction result for Andersen dynamics on $\mathbb{T}_{\ell}^{\m}\times \mathbb{R}^{\m}$}. Let $(p_t)_{t \ge 0}$ denote the transition semigroup.
The parameters defining the coupling and the metric are defined in the following way: 
 \begin{eqnarray}
 \mathcal{R} &=& \ell/2+\m/(\beta^{1/2} \lambda) \;, \label{eq::R} \\
\gamma &=& 1/(\beta^{1/2} \mathcal{R})   \;,  \label{eq::gamma} \\
a &=& \beta^{1/2} \lambda / \m \;, \quad \text{and}  \label{eq::a}  \\
\alpha &=& \sqrt{1 + \beta L \mathcal{R}^2}  \;.    \label{eq::alpha}   
 \end{eqnarray}
The choice of $\gamma$ is motivated by Lemma \ref{lem:rejp}, and the choice of the other parameters is motivated by the proof of the following theorem.

\begin{figure}[t]   
    \centering
\begin{minipage}{\textwidth}
\begin{minipage}{0.45\textwidth}
\begin{center}
\begin{tikzpicture}[scale=1.0]
    \draw[draw=blue,->-=.5,thick] (-2,1) to (0,1);
    \draw[draw=blue,->-=.5,thick] (0,1) to (2,1);
    \draw[draw=blue,->-=.5,thick] (2,-1) to (0,-1);
        \draw[draw=blue,->-=.5,thick] (0,-1) to (-2,-1);
              \draw[scale=1.0,domain=-2.0:2.0,variable=\x,red,thick] plot[samples=200] ({\x},{-0.5*\x}) node[below] {\large $q=0$};
      \draw[->] (-2.5,0) -- (2.5,0) node[right] {\large
$\zeta$};
      \draw[->] (0,-1.5) -- (0,1.5) node[above] {\large
 $w$};
      \draw[-] (-2,-0.1) -- (-2,0.1) node[below=0.15cm] {\normalsize
$-\ell/2$};
      \draw[-] (2,-0.1) -- (2,0.1) node[below=0.15cm] {\normalsize
$\ell/2$};
          \filldraw[color=red,fill=white] (2.0,1) circle (0.07);
        \filldraw[color=blue,fill=blue] (2.0,-1) circle (0.07);
        \filldraw[color=red,fill=white] (-2.0,-1) circle (0.07); 
                \filldraw[color=blue,fill=blue] (-2.0,1) circle (0.07); 
    \end{tikzpicture}
    \end{center}
\end{minipage}
\begin{minipage}{0.45\textwidth}
\begin{center}
\begin{tikzpicture}[scale=1.0]
      \draw[->] (0,0) -- (5,0) node[right] {\large
$t$};
      \draw[->] (0,0.0) -- (0,2.5) node[above] {\large
$r_t$};
            \draw[-] (1,0.1) -- (1,-0.1) node[below] {\normalsize
 $t_1$};
      \draw[-] (-0.1,1) -- (0.1,1) node[left=0.2cm] {\normalsize
$\ell/2$};
\foreach \a in {0,...,4}{
      \draw[scale=1.0,domain=\a:{\a+1},variable=\x,blue,thick] plot[samples=200] ({\x},{(\x-\a)+1});
                       \draw[dashed,blue] ({\a+1},1) -- ({\a+1},2);
                       \filldraw[color=red,fill=white] ({\a+1},2) circle (0.07);
        \filldraw[color=blue,fill=blue] ({\a+1},1) circle (0.07);
        }
    \end{tikzpicture}
    \end{center}
\end{minipage}
\end{minipage}
\begin{minipage}{\textwidth}
\begin{minipage}[b]{0.48\textwidth}
      \centering
(a)  
\end{minipage}
\begin{minipage}[b]{0.48\textwidth}
      \centering
(b) 
\end{minipage}
\end{minipage}
\caption{\small Plots of a piecewise-constant-velocity trajectory initially at $\zeta_0=-\ell/2$ and $w_0>0$ that jumps at $t=t_1$ to $w_{t_1}=-w_0$  where $\zeta_{t_1}=\ell/2$ in (a) phase space and (b) $r_t$ as a function of time. Note that  $t \mapsto r_{t}$ is a c\`{a}dl\`{a}g trajectory. }
\label{fig:r}
\end{figure}
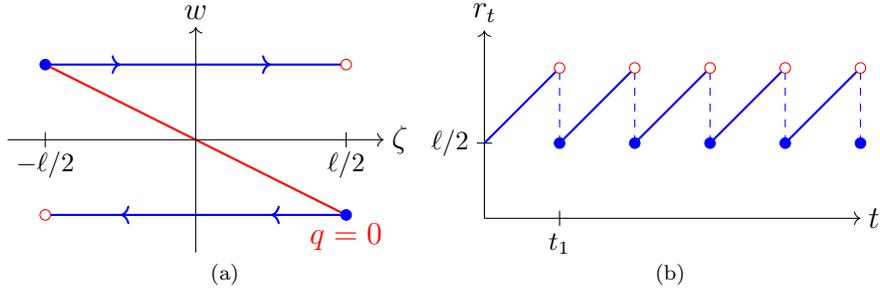

%
%

\begin{theorem} \label{TOR}
Suppose that Assumption~\ref{hypo::U} holds and let $\lambda>0$ satisfy
\begin{equation} \label{eq:lambcond_tor}
\beta^{1/2} \,  \frac{\lambda}{\m} \, \frac{\ell}{2}  \ \ge \  \frac{25}{6} + 11 \, \beta \,  L \, \left(\frac{\ell}{2}\right)^2 \;.
\end{equation} 
Suppose moreover that 
\begin{equation}
    \label{eq:Jbound}
   {J}\ \le\ \frac{1}{75(\m -1)\beta\ell^2}\max\left(\sqrt{\beta L\ell^2},1\right)
    \exp\left(- \beta^{1/2} \frac{ \lambda}{\m} \frac{\ell}{2}\right).
\end{equation}
Then for all $y \in  \mathbb{T}_{\ell}^{\m} \times \mathbb{R}^{3 \m} $, the process $e^{\cA \, t} \rho(Y_t)$ is a nonnegative supermartingale where
\begin{equation}
 \label{eq:rate_tor}
\cA   :=   \frac{1}{90} \frac{\lambda}{\m} \exp\left(- \beta^{1/2} \frac{ \lambda}{\m} \frac{\ell}{2}\right) 
\end{equation}
Moreover, for all probability measures $\nu$ and $\mu$ on $\mathbb{T}_{\ell}^{\m} \times \mathbb{R}^{\m}$ we have \begin{equation} \label{eq::Wrho_tor}
\mathcal{W}_{\rho}(\mu p_t, \nu p_t) \ \le \ e^{-\cA \, t} \, \mathcal{W}_{\rho}(\mu , \nu) \;.  
\end{equation}
\end{theorem}

A proof of this theorem is provided in Section~\ref{sec::proofs}.  Remarkably, the result captures the correct order of the dimension dependence for Andersen dynamics in the free-streaming case where $L=J=0$, and condition~\eqref{eq:lambcond_tor} reduces to $\beta^{1/2} (\lambda/\m) (\ell/2) \ge 25/6$. A corresponding bound holds for weak interactions, i.e., when $J$ satisfies 
Condition \eqref{eq:Jbound}. On the other hand, a restriction on $J$ can not be avoided. Indeed, for large
values of $J$,  phase transitions in the infinite dimensional limit can cause an exponential degeneration 
of the rate of convergence to equilibrium as the number 
$m$ of particles goes to infinity, even if $\lambda $ is increased linearly with $\m$.



\section{Proofs in the weakly anharmonic case} \label{sec:proofs}

\begin{proof}[Proof of Theorem~\ref{WAH}]
Here we apply the synchronous coupling described in Remark~\ref{rmk:sync}. Let $(Z_t, W_t) := (X_t - \tilde X_t, V_t - \tilde V_t)$.  In between two consecutive jump times, $t \in [T_{k}, T_{k+1})$, note that the time derivative of $(Z_t, W_t)$ satisfies
\[
\frac{d}{dt} Z_t \ =\ W_t , \quad  \frac{d}{dt} W_t \ = \ - \mathcal{C}^{-1} Z_t - (\nabla G(X_t) - \nabla G(\tilde X_t) ) , 
\] with $Z_{T_{k}}=Z_{T_{k}-}$, $W_{T_{k}}^{I_{k}} = 0$, and $W_{T_{k}}^{j} = W_{T_{k}-}^{j}$ for $j \ne I_{k}$.  In particular, when $G \equiv 0$ these differential equations become Hamiltonian with respect to the unperturbed Hamiltonian function $H_0(z,w)$, and hence for $y=(x,v,\tilde x,\tilde v)$, $Z=x-\tilde x$ and $w=v-\tilde v$,  \begin{eqnarray}
\label{eq:LC_H0}
(\mathcal{L}^C H_0)(y) \ = \ - (\nabla G(x) - \nabla G(\tilde x)) \cdot w  \ \le \  L_G |w| |z| 
\end{eqnarray}
where we applied in turn the definition of $\mathcal{L}^C$ in \eqref{eq:LC} and \eqref{eq:wah_lip} in Assumption~\ref{hypo:U}. Similarly, applying $\mathcal{L}^C$ to $\Psi (y):=z \cdot w$ gives  \begin{eqnarray}
\label{eq:LC_virial}
(\mathcal{L}^C\Psi ) ( y )  \ = \   |w|^2 - 
z \cdot ( \nabla U(x) - \nabla U(\tilde x)) \ \le \ |w|^2 - z^T \mathcal{C}^{-1} z 
\end{eqnarray}
where in the last step we used convexity of the perturbation $G(x)$.
Applying $\mathcal{L}^C$ to $\rho^2$ in \eqref{eq:rho}, and then inserting \eqref{eq:LC_H0} and \eqref{eq:LC_virial} yields
  \begin{eqnarray} 
(\mathcal{L}^C \rho^2)(y) & \le & \frac{\lambda}{4 \m}|w|^2  + \frac{\lambda^2}{4 \m^2} z \cdot w - \frac{\lambda}{4 \m} z^T \mathcal{C}^{-1} z + L_G |w| |z| \;. \label{eq:LC_wah}
\end{eqnarray}

By definition of $\mathcal{A}^C_{sync}$ in \eqref{eq:AC} and Remark \ref{rmk:sync}, \begin{eqnarray}
(\mathcal{A}^C_{sync} \, \rho^2)(y) & = & - \frac{\lambda^2}{4 \m^2} z \cdot w - \frac{\lambda}{2 \m} |w|^2 \;. \label{eq:AC_wah}
\end{eqnarray}

Combining \eqref{eq:AC_wah} and \eqref{eq:LC_wah} yields \begin{eqnarray}
(\mathcal{G}^C_{sync} \, \rho^2)(y)  &\le & -\frac{\lambda}{4 \m} \left( |w|^2 + z^T \mathcal{C}^{-1} z - \frac{4 L_G \m}{\lambda} |w| |z| \right) \nonumber \\
& \le & -\frac{\lambda}{4 \m} \left( \frac{1}{2} |w|^2 + \left( z^T \mathcal{C}^{-1} z - \frac{8 L_G^2 \m^2}{\lambda^2} |z|^2 \right)  \right) \nonumber \\
& \le & - \frac{\lambda}{4 \m} H_0(z,w)   \label{eq:GC_wah}
\end{eqnarray}
where in the last step we applied condition \eqref{eq:lambcond_wah} and $\sigma^{-2}_{max} |z|^2 \le z^T \mathcal{C}^{-1} z$ which together imply that $(8 L_G^2 \m^2 / \lambda^2) |z|^2 \le (1/2) \sigma_{max}^{-2} |z|^2 \le (1/2) z^T \mathcal{C}^{-1} z$.  Note that   
\begin{eqnarray}
     \rho(y)^2 &= &  H_0(z,w)  + \frac{\lambda}{4 \m} z \cdot w + \frac{\lambda^2}{8 \m^2}  |z|^2 \
      \ \le \ 2 H_0(z,w)+   \frac{5 \lambda^2}{32 \m^2} |z|^2  \nonumber \\
    & \le &  \max\left( 2, \frac{5 \sigma_{max}^2 \lambda^2}{32 \m^2}  \right)  H_0(z,w)  \label{eq:comp_wah} 
\end{eqnarray}
where in the last step we again used $\sigma^{-2}_{max} |z|^2 \le z^T \mathcal{C}^{-1} z$.  Inserting \eqref{eq:comp_wah} into \eqref{eq:GC_wah} gives the required infinitesimal contraction result in \eqref{WAH1}.  

\medskip

For the corresponding Wasserstein bound, first, note from \eqref{eq:rhoG} 
 \begin{equation} \label{eq:comp_wah_2}
\lambda_{min}(\G) (|z|^2 + |w|^2) \ \le \  \rho(y)^2 \ \le \ \lambda_{max}(\G) (|z|^2 + |w|^2) \;,
\end{equation}
where $\lambda_{min}(\G)$ and $\lambda_{max}(\G)$ are the smallest and largest eigenvalues of the matrix $\G$, respectively.  Let $g(t,y):=e^{c t} \rho(y)^2$. Then by \eqref{WAH1},$\frac{\partial g}{\partial t}+\mathcal{G}^C_{sync}g\le 0 $. Hence by \cite[Theorem 5.5]{davis1984piecewise}, the process $g(t,Y_t)$ is a non-negative supermartingale, and thus, $\E\left[ \rho(Y_t)^2 \right] \le e^{-c t} \rho(y)^2$.  Therefore, by the coupling characterization of the 2-Wasserstein metric and \eqref{eq:comp_wah_2}, \begin{align*}
\mathcal{W}_2(\mu p_t, \nu p_t)^2 \ \le \  \lambda_{min}(\G)^{-1} \E\left[ \rho(Y_t)^2 \right] \  \ \le \ 
\varkappa(\G) e^{-c t} \mathcal{W}_2(\mu, \nu)^2 \;,
\end{align*}
where $\varkappa(\G) = \lambda_{max}(\G) \lambda_{min}(\G)^{-1}$ is the condition number of $\G$.
By taking square roots, we obtain the required bound in \eqref{eq:W2_wah}.
\end{proof}

%
%
%


\section{Proofs for Andersen dynamics on a high-dimensional torus} \label{sec::proofs}

To prove contractivity of Andersen dynamics on $\mathbb{T}_{\ell}^{\m}$, and as illustrated in Figure~\ref{fig:ri_sublevel}, we use the distance function $r_i(y)$ in \eqref{eq:ri} to decompose $\mathbb{T}_{\ell}^{\m} \times \mathbb{R}^{3 \m}$ into the following sets: $\{ r_i > \mathcal{R} \}$,  $\{ 0<r_i \le \mathcal{R} \}$,  and $Z_i \ := \ \{ r_i=0 \}$.  In addition, we introduce the following subset \begin{align}
    B_i & \ := \ \{ y=(x,v,z,w) \in \mathbb{T}_{\ell}^{\m} \times \mathbb{R}^{3 \m} ~:~ \zeta_i(z,w)=\ell/2 ~~\text{and}~~ w_i = 0 \}  \;. 
\end{align}
The following remark shows that with the definition of  $\mathcal{R}$ in \eqref{eq::R}, and under condition \eqref{eq:lambcond_tor},  $B_i \subset \{ r_i > \mathcal{R} \}$, see also Figure \ref{fig:ri_sublevel}.

\begin{remark} \label{Bi}
In Theorem~\ref{TOR}, the condition on $\lambda$ in \eqref{eq:lambcond_tor} implies \[
\beta^{1/2} \lambda \mathcal{R} / \m \ge 4 + 6 \beta L \mathcal{R}^2 \;.
\] Under this condition, by definition of $\mathcal{R}$ in \eqref{eq::R}, \[
\frac{\mathcal{R} - \ell/2}{\mathcal{R}} = \frac{\m}{\beta^{1/2} \lambda \mathcal{R}} \le \frac{1}{4 + 6 \beta L \mathcal{R}^2} 
\] which implies that $\ell/(2 \mathcal{R}) \ge (3+6 \beta L \mathcal{R}^2)/(4 + 6 \beta L \mathcal{R}^2) $ and hence \[
\frac{\ell}{2} < \mathcal{R} \le \frac{4 + 6 \beta L \mathcal{R}^2}{3 + 6 \beta L \mathcal{R}^2} \frac{\ell}{2}  = \frac{\ell}{2} + \frac{1}{3+6 \beta L \mathcal{R}^2} \frac{\ell}{2} 
\] In particular, $\mathcal{R} \le (4/3) \ell/2$, and consequently, for all $y \in B_i$,
\begin{align*}
    r_i(y) \ &= \ \sqrt{1 + \alpha^{-2}} \,  |\zeta_i| = \sqrt{1+\alpha^{-2}} \, (\ell/2) \ \ge \  ( 1 + (\sqrt{2} - 1) \alpha^{-2} ) \, (\ell/2)  \\
    \ &\ge \  \left(1 + \frac{\sqrt{2}-1}{1+\beta L \mathcal{R}^2} \right) \, \frac{\ell}{2} \ge \left(1 + \frac{1/3}{1+\beta L \mathcal{R}^2} \right) \, \frac{\ell}{2}  \ge \mathcal{R} \;,
\end{align*} 
where we used the inequality $\sqrt{1+\mathsf{x}} \ge 1 + (\sqrt{2}-1) \mathsf{x}$ valid for all $\mathsf{x} \in [0,1]$, and \eqref{eq::alpha} to eliminate $\alpha$.  
\end{remark}

By Remark~\ref{Bi}, for all $y \in B_i$, $f(r_i(y))=f(\mathcal{R})$ is constant and maximal.  As we will see below, this observation simplifies the bounds on the metric along the deterministic flow of \eqref{odes:tor} starting at $y \in B_i$.

\begin{figure}[ht!]      
\begin{minipage}{0.55\textwidth}
\includegraphics[width=\textwidth]{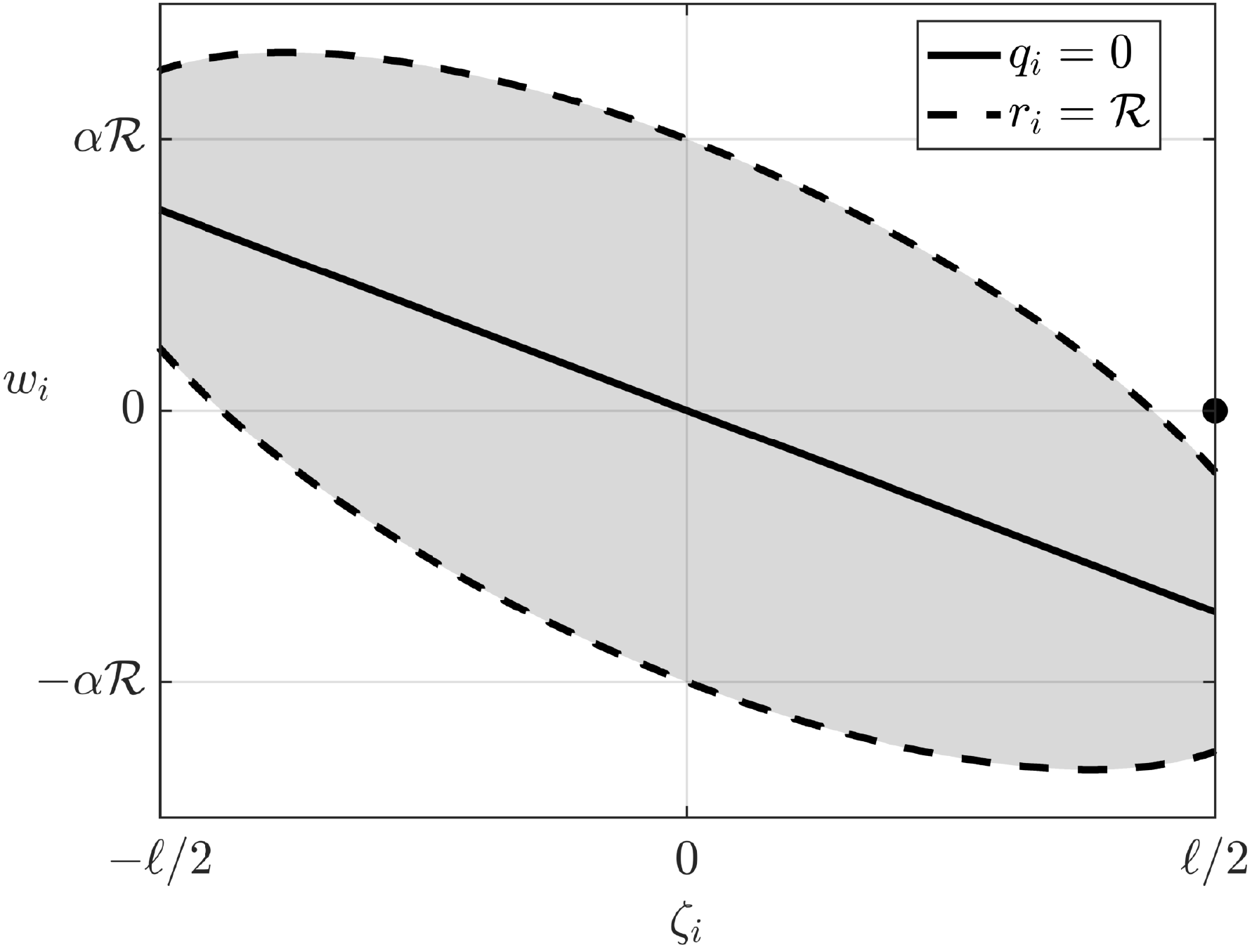}
 \end{minipage} \hfill
   \begin{minipage}{0.4\textwidth} 
 
\caption{ \small  
The grey-shaded region corresponds to the sublevel set $\{ r_i \le \mathcal{R} \}$ whose boundary is the truncated ellipse indicated by the dashed black line. The solid black line corresponds to the line segment $q_i = \zeta_i + \gamma^{-1} w_i =0$, which is on the long axis of the truncated ellipse. 
As noted in Remark~\ref{Bi}, the point $(\zeta_i, w_i)=(\ell/2,0)$ is not in the grey-shaded region.   
 }    \label{fig:ri_sublevel}
 \vspace{1cm}
\end{minipage} 
\end{figure}

\subsection{Bounds for coupling of velocities} \label{proof:rejp}

\begin{proof}[Proof of Lemma~\ref{lem:rejp}]
Let $\mathbf{1}_{\n}$ be the $\n \times \n$ identity matrix and
introduce $w' = \xi - \tilde \xi$. Noting that  $\Pr( w' \ne - \gamma \mathsf{b} ) = d_{\TV}(\mathcal{N}(0,\beta^{-1} \mathbf{1}_{\n}), \mathcal{N}(\gamma \mathsf{b},\beta^{-1} \mathbf{1}_{\n}))$ \cite[Section 2.3.2]{BoEbZi2020}, scale invariance of the total variation distance implies
\begin{align*}
 \Pr(  w' \ne - \gamma \mathsf{b} ) &= d_{\TV}(\mathcal{N}(0,\mathbf{1}_{\n}), \mathcal{N}(\sqrt{\beta} \gamma \mathsf{b},\mathbf{1}_{\n}))  \
  =\ d_{\TV}(\mathcal{N}(0,1), \mathcal{N}(\sqrt{\beta} \gamma |\mathsf{b}|,1))  \\
& = 2 \mathcal{N}(0,1) \left[ (0,\sqrt{\beta} \gamma |\mathsf{b}|/2) \right]  \le \sqrt{\beta}  \gamma |\mathsf{b}| / \sqrt{2 \pi}  \;. 
 \end{align*}
{Hence \eqref{ieq:rejpA} holds.} Figure 4 of \cite{BoEb2019} illustrates the second to last step.  
 
 \smallskip
 
When $\mathsf{b} = 0$, the set $\{ w' \ne - \gamma \mathsf{b} \}$ is empty and \eqref{ieq:rejpB} holds.  Thus, suppose that $\mathsf{b} \ne 0$. Then the set $\{ x \in \mathbb{R}^{\n} : x \cdot \mathsf{b} = 0 \}$ defines an $\n-1$ dimensional hyperplane. By \eqref{eq:tildexi},
\begin{align*}
 &\mathbb{E} \left( |\xi|^2 ; ~w' \ne - \gamma \mathsf{b} \right) = \int_{\mathbb{R}^{\n}} |x|^2  ( \varphi_{\beta}(x) - \varphi_{\beta}(x) \wedge \varphi_{\beta}(x+\gamma \mathsf{b}) ) dx \\
 &= \int_{\mathbb{R}^{\n}} |x|^2 ( \varphi_{\beta}(x) - \varphi_{\beta}(x + \gamma \mathsf{b}) )^+ d x \\
 &= \int_{\mathbb{R}^{\n}} |x - (1/2) \gamma \mathsf{b}|^2 ( \varphi_{\beta}(x - (1/2) \gamma \mathsf{b}) - \varphi_{\beta}(x + (1/2) \gamma \mathsf{b}) )^+ d x \\
& = \int_{\{\mathsf{b} \cdot x \ge 0\}} |x - (1/2) \gamma \mathsf{b}|^2 ( 1- e^{ - \beta \gamma \mathsf{b} \cdot x } ) \varphi_{\beta}( x - (1/2) \gamma \mathsf{b}) dx = \rn{1} + \rn{2}
\end{align*}
where we introduced $\rn{1}$ and $\rn{2}$ \begin{align}
    \rn{1} &=\int\limits_{\{\mathsf{b} \cdot x_{\parallel} \ge 0\}}  |x_{\perp}|^2  ( 1- e^{ - \beta \gamma \mathsf{b} \cdot x_{\parallel} } ) \varphi_{\beta}(x_{\parallel} + x_{\perp} -\frac{1}{2} \gamma \mathsf{b}) d x_{\parallel} d x_{\perp} \;, \\
    \rn{2} &= \int\limits_{\{\mathsf{b} \cdot x_{\parallel} \ge 0\}}  |x_{\parallel} - \frac{1}{2} \gamma \mathsf{b}|^2  ( 1- e^{ - \beta \gamma \mathsf{b}\cdot x_{\parallel}  } ) \varphi_{\beta}(x_{\parallel} + x_{\perp} -\frac{1}{2} \gamma \mathsf{b}) d x_{\parallel} d x_{\perp} \;,  \label{rejp_2}
\end{align}
 that involve a change of variables given by $x = x_{\parallel} + x_{\perp}$ with $x_{\perp} \cdot \mathsf{b} = 0$.

Now let $\phi(s) = \exp(-(1/2) \beta |s|^2 ) / \sqrt{2 \pi \beta^{-1}}$. Integration over the $x_{\perp}$ variable yields $\rn{1} = \beta^{-1} (\n - 1)  \Pr( w' \ne - \gamma \mathsf{b} )$ and \begin{align}
& \rn{2} = \int_0^{\infty} ( 1- e^{ - \beta \gamma |\mathsf{b}| s } ) | s -  \gamma |\mathsf{b}|/2|^2   \phi(s - \gamma |\mathsf{b}|/2) ds \nonumber \\
&= \int_0^{\infty} ( 1- e^{ - \beta \gamma |\mathsf{b}| s } ) \left(  \beta^{-2} \frac{d^2}{ds^2} \phi(s - \gamma |\mathsf{b}|/2) + \beta^{-1} \phi(s - \gamma |\mathsf{b}|/2) \right) ds \nonumber  \\
&=\int_0^{\infty} ( 1- e^{ - \beta \gamma |\mathsf{b}| s } )   \beta^{-2} \frac{d^2}{ds^2} \phi(s - \gamma |\mathsf{b}|/2) ds  + \beta^{-1} \Pr(w' \ne - \gamma \mathsf{b}) \nonumber  \\
&= - \gamma^2  |\mathsf{b}|^2 \int_0^{\infty}   \phi(s + \gamma |\mathsf{b}|/2) ds + \frac{ \gamma |\mathsf{b}| \phi( \gamma |\mathsf{b}| / 2)  +  \Pr(w' \ne - \gamma \mathsf{b}) }{\beta} \nonumber \\
&= - \frac{\gamma^2  |\mathsf{b}|^2}{2} \Pr(w' = - \gamma \mathsf{b}) + \frac{ \gamma |\mathsf{b}| \phi( \gamma |\mathsf{b}| / 2)  +  \Pr(w' \ne - \gamma \mathsf{b}) }{\beta} \label{rejpb:II}
\end{align}
where in the last step integration by parts was used twice.
Combining \rn{1} and \rn{2} with \eqref{ieq:rejpA} gives $\mathbb{E} \left( |\xi|^2 ; ~w' \ne - \gamma \mathsf{b} \right) \le (\n + 1)   \gamma |\mathsf{b}| / \sqrt{2 \pi \beta} $; hence, \eqref{ieq:rejpB} holds.  
\end{proof}

\subsection{Bounds for Andersen collision operator acting on metric}

\begin{lemma} \label{AC_tor_lemma}
Suppose that $\lambda>0$ satisfies condition~\eqref{eq:lambcond_tor}.
For any $i \in \{1, \dots, \m\}$, 
\begin{equation}
    \mathcal{A}^C_{\gamma} ( f \circ r_i )  \ \le  \ \begin{cases} \ - \frac{9}{25} \frac{\lambda}{\m} \exp\left(- \beta^{1/2} \frac{ \lambda}{\m} \frac{\ell}{2} \right) f\circ r_i & \text{if $r_i > \mathcal{R}$} \;, \\
   - \gamma \left( -\frac{2}{5} |\zeta_i|^2 + \left( \frac{3}{10} \frac{\lambda}{\gamma \m} - \frac{1}{5} \right) \right) \frac{f_-'\circ r_i}{r_i}  & \text{if $0 < r_i \le  \mathcal{R}$} \;, \\
   \  0 & \text{if $r_i=0$}  \;.
    \end{cases}
\end{equation}
\end{lemma}

\begin{proof}
Fix $y \in \mathbb{T}^{\m}_{\ell} \times \mathbb{R}^{3 \m}$.
Let $I \sim \Unif\{1,\dots,\m\}$, $\xi \sim \mathcal{N}(0,\beta^{-1})^{\n}$ and $\mathcal{U} \sim \Unif(0,1)$ 
be independent random variables. We set
$\tilde \xi = \Phi(\xi, \zeta_i, \mathcal{U})$ and introduce the shorthand $w'_i = \xi-\tilde \xi$.  
Since $\gamma = \beta^{-1/2} \mathcal{R}^{-1}$ by \eqref{eq::gamma}, 
\begin{equation} \label{eq:rejp}
 \Pr( w'_i \ne - \gamma \zeta_i )\  \le\ {\sqrt{\beta} \gamma |z_i|}/{\sqrt{2 \pi}}\  \le\ {\sqrt{\beta} \gamma \mathcal{R}}/{\sqrt{2 \pi}}\  \le\  {1}/{\sqrt{2 \pi}} \ < \ 2/5 \;.
\end{equation}

\paragraph{Bound for $r_i(y) > \mathcal{R}$.}
On $[\mathcal{R}, \infty)$, $f$ is constant and takes its maximum value.  Therefore, $f(r_i(y)) = f(\mathcal{R})$ and \begin{align}
    \mathcal{A}^C_{\gamma} & (f \circ r_i)( y ) \  = \ \lambda \E(f ( r_i ( \mathcal{S}(I,\xi,\mathcal{U}) y) ) - f( r_i(y) )  )  \nonumber \\
    \ & \le \ \frac{\lambda}{\m} \E\left( f(r_i(\mathcal{S}(i,\xi,\mathcal{U}) y)) - f(r_i(y) ) ; ~ w'_i = - \gamma \zeta_i \right) \label{ACbnd:rgeR}\\
    & \le \ \frac{\lambda}{\m} ( f(|\zeta_i|) - f(r_i) ) \Pr(w'_i = - \gamma \zeta_i) 
    \ \le \ - \frac{3}{5} \frac{\lambda}{\m} \left( 1 - \frac{f(\ell/2)}{f(\mathcal{R})} \right) f(r_i(y))  \nonumber
\end{align}
where in the last step we used $\mathcal{R} \ge \ell/2$ and \eqref{eq:rejp} which implies that $\Pr(w'_i = - \gamma \zeta_i) \ge 3/5$.
Since, by \eqref{eq::R}, $\mathcal{R}=1/a+\ell/2$, and using $1-e^{-1} \ge 3/5$, we have  \[
1 - f(\ell/2)/f(\mathcal{R}) = (1-e^{-1})/(e^{a \ell/2}-e^{-1}) \ge (3/5) e^{-a \ell/2} \;.
\] Inserting this inequality back into \eqref{ACbnd:rgeR} and eliminating $a$ using \eqref{eq::a} gives the required bound.  

\paragraph{Bound for $0 < r_i(y) \le \mathcal{R}$.} Let $r_i'=r_i(\mathcal{S}(i,\xi,\mathcal{U}) y)$ and write   \begin{equation} \label{ACbnd:rleR}
\mathcal{A}^C_{\gamma}  (f \circ r_i) = \rn{1} + \rn{2}~~\text{where}~~\begin{cases}
\rn{1} := \frac{\lambda}{\m} \E\left( f(r_i') - f(r_i ) ; ~ w'_i = - \gamma \zeta_i \right) \\
\rn{2} := \frac{\lambda}{\m} \E\left( f(r_i') - f(r_i ) ; ~ w'_i \ne - \gamma \zeta_i \right)
\end{cases}
\end{equation} For $\rn{1}$, note that on $\{ w_i' = - \gamma \zeta_i \}$,
\[
r_i' - r_i \ = \ |\zeta_i|-r_i\ =\  (|\zeta_i|^2 - r_i^2)/(|\zeta_i| + r_i) \ \le \ -\alpha^{-2} |q_i|^2 / (2 r_i) \;.
\] 
Combining this bound with \eqref{ieq:f} and
\eqref{eq:rejp}, we obtain \begin{equation} \label{ACbnd:rleR_1}
    \rn{1} \ \le \ - \frac{\lambda}{\m} \alpha^{-2} |q_i|^2 \frac{f_-'(r_i)}{2 r_i} \Pr(w_i' = - \gamma \zeta_i) \ \le \  - \frac{3}{10} \frac{\lambda}{\m} \alpha^{-2} |q_i|^2 \frac{f_-'(r_i)}{r_i} 
\end{equation}
For $\rn{2}$, use \eqref{ieq:f},  \eqref{eq:rejp} and \eqref{eq::a} to obtain \begin{align}\nonumber
 \rn{2}  &\le  \frac{\lambda}{\m} a^{-1} f_-'(r_i) \Pr( w_i' \ne - \gamma \zeta_i)  \le  \frac{\lambda}{\m} a^{-1} f_-'(r_i) \frac{\sqrt{\beta} \gamma |\zeta_i|}{\sqrt{2 \pi}}\\ & = \frac{\gamma f_-'(r_i) |\zeta_i| r_i}{ r_i \sqrt{2 \pi}}      
  \le \frac{2}{5} \gamma \left( |\zeta_i|^2 + \frac{\alpha^{-2}}{2} |q_i|^2 \right) \frac{f_-'(r_i)}{r_i} \;. \label{ACbnd:rleR_2}
\end{align}
Inserting \eqref{ACbnd:rleR_1} and \eqref{ACbnd:rleR_2} into \eqref{ACbnd:rleR} gives the required bound.  

\paragraph{Bound for $r_i(y)=0$.} In this case, $\zeta_i=w_i=0$, and thus, $\mathcal{S}(i, \xi, \mathcal{U}) y = y$, i.e., $\mathcal{A}^C_{\gamma} (f \circ r_i)(y) = 0$, as required.  
\end{proof}

\subsection{Regularity of distance function under flow of \eqref{odes:tor}}
\label{sec::proofs:regularity}

Here we prove Lemma~\ref{lem:cadlag} --- a key ingredient to controlling boundary effects for $|\zeta^i|=\ell/2$ and $w^i \ne 0$. The following remark is useful in the proof.

\noindent
\begin{minipage}{\textwidth}
\begin{minipage}{0.45\textwidth}
\begin{center}
\begin{tikzpicture}[scale=1.0]
      \draw[->] (-2.5,0) -- (2.5,0) node[right] {\large
$z_i$};
      \draw[->] (0,0) -- (0,1.75) node[right] {\large
$|\zeta_i|$};
            \draw[-] (-2,0.1) -- (-2,-0.1) node[below=0.1cm] {\normalsize
 $-\ell$};
      \draw[-] (-1,0.1) -- (-1,-0.1) node[below=0.1cm] {\normalsize
$-\ell/2$};
      \draw[-] (0,0.1) -- (0,-0.1) node[below=0.1cm] {\normalsize
$0$};
      \draw[-] (1,0.1) -- (1,-0.1) node[below=0.1cm] {\normalsize
$\ell/2$};
      \draw[-] (2,0.1) -- (2,-0.1) node[below=0.1cm] {\normalsize
$\ell$};
      \draw[-] (-0.1,1) -- (0.1,1) node[left=0.1cm] {\normalsize
$\ell/2$};
\foreach \a in {0,...,2}{
            \draw[scale=1.0,domain=(-2.5):(2.5),variable=\x,blue,thick] plot[samples=100] ({\x},{abs(\x-floor((\x+1)/2)*2)});
         }        
    \end{tikzpicture}
    \end{center}
\end{minipage}
\begin{minipage}{0.55\textwidth}
\noindent
As illustrated to the left, $(z_i, w_i) \mapsto |\zeta_i|$ is a contraction in the sense that 
\begin{equation} \label{zeta_i_contr}
\left| |\zeta_i| - |\tilde \zeta_i| \right| \le | z_i - \tilde z_i | \;. 
\end{equation} 
\end{minipage}
\end{minipage}

\begin{proof}[Proof of Lemma~\ref{lem:cadlag}]
Let $r^i_t = r_i(y_t)$ and $\zeta^i_t = \zeta_i(z_t,w_t)$ where $i \in \{1, \dots, \m\}$ is fixed and $y_t$ is the deterministic solution of \eqref{odes:tor} starting at $y_0=y$. Recall from \eqref{eq:ri} that $r^i_t = \sqrt{|\zeta^i_t|^2 + \alpha^{-2} |\zeta^i_t + \gamma^{-1} w^i_t|^2}$.  The function $t \mapsto (z^i_t, w^i_t)$ is continuous.  Moreover, $\zeta_i$ in \eqref{zeta_i} is continuous at points where $|\zeta_i| < \ell/2$.  Therefore, if $|\zeta^i_t| < \ell/2$, then $\zeta^i_t$ and $r^i_t$ are continuous at $t$.

Suppose, next, that at time $t$, $\zeta^i_t=-\ell/2$ and $w^i_t>0$.  Since $w^i_t > 0$, $z^i_t$ is strictly increasing in an open interval containing $t$.  Therefore, for sufficiently small $h>0$, \[
\zeta^i_{t+h} - \zeta^i_t = z^i_{t+h} - z^i_t \;, \quad \text{and} \quad \zeta^i_{t}  - \zeta^i_{t-h} = (z^i_{t} - \ell) - z^i_{t-h}  \;.
\] Hence, $\lim_{h \downarrow 0} (\zeta^i_{t+h} - \zeta^i_t) = 0$ while $\lim_{h \downarrow 0} (\zeta^i_{t} - \zeta^i_{t-h}) = -\ell$, and in particular, $(\zeta^i_t, w^i_t)$ is c\`{a}dl\`{a}g, and hence, $r^i_t$ is c\`{a}dl\`{a}g as well. Moreover, $(z_i, w_i) \mapsto |\zeta_i|$ is continuous because $(z_i, w_i) \mapsto |\zeta_i|$ is a contraction by \eqref{zeta_i_contr}, and thus, \begin{align*}
    |q^i_t|^2 \ = \ |\zeta^i_t + \gamma^{-1} w^i_t|^2 \ = \   |\zeta^i_t |^2 + \gamma^{-2} |w^i_t|^2 + 2 \gamma^{-1} \zeta^i_t w^i_t \ \le \ \lim_{s \uparrow t} |q^i_s|^2 
\end{align*}
because $w^i_t >0$ and $\zeta^i_t \le \zeta^i_{t-}$.  Therefore, $r^i_t \le \lim_{s \uparrow t} r^i_s$.  The case $\zeta^i_t = \ell/2$ and $w^i_t < 0$ can be treated similarly; in this case $\lim_{h \downarrow 0} ( \zeta^i_t - \zeta^i_{t-h}) = \ell$.

Finally, suppose that at time $t$, $\zeta^i_t = \ell/2$ and $w^i_t = 0$.  In this case, $r_i$ is itself continuous at $y_t$, and therefore, $s \mapsto r^i_s = r_i(y_s)$ is continuous at $t$.   Continuity of $r_i$ at $y_t$ follows from $|q_i(z,w)| \to |\zeta_i(z,w)|$ as $w_i \to 0$.   
\end{proof}

\begin{remark}
By Lemma~\ref{lem:cadlag}, $t \mapsto r_t^i$ is a c\`{a}dl\`{a}g trajectory.  Therefore, for any $\epsilon>0$ and for any $T>0$, the number of jumps of size greater than $\epsilon$, i.e.,  $\# \{ t \in [0,T] : |r^i_t - r^i_{t-}|>\epsilon \}$, is finite \cite{EberleLectureNotes2019}.  However, for a trajectory starting in $B_i$ where $(\zeta_0^i, w_0^i)=(\ell/2,0)$, it is still possible that there are infinitely many jumps in every interval $(0,h)$ with $h>0$, i.e., the underlying trajectory $t \mapsto (\zeta_t^i,w_t^i)$ may wind around the point $(\ell/2,0)$ infinitely often.  For the bounds on the deterministic part of the dynamics, we avoid this potential complication by selecting $\mathcal{R}$ and $\lambda$ such that $B_i \subset \{ r_i(y) > \mathcal{R} \}$ where $f(r_i(y)) = f(\mathcal{R})$ is constant and maximal; see Remark~\ref{Bi}.  
\end{remark}

\subsection{Bounds for Liouville operator acting on metric}

Since $r_i$ in \eqref{eq:ri} lacks continuity at boundary points where $|\zeta_i|=\ell/2$, the domain of $\mathcal{L}^C$ excludes $\rho$.  Nonetheless, by Lemma~\ref{lem:cadlag}, $t \mapsto r_t^i$ is a c\`{a}dl\`{a}g trajectory.  This c\`{a}dl\`{a}g time regularity motivates defining the following right-sided directional derivative of a function along the deterministic flow of \eqref{odes:tor}.

\begin{definition} \label{right_sided_LC}
For a function $g:\mathbb{T}^{\ell}_{\m} \to \mathbb{R}$, define \[
\mathscr{L}^C g(y) \ := \ \lim_{h \downarrow 0} \dfrac{g(\phi_h^C(y))-g(y)}{h} \quad \text{whenever the limit exists.}
\] 
\end{definition}

According to this definition, $\mathscr{L}^C ( f \circ r_i)$ is well-defined at most boundary points, and in particular,  \begin{equation} \label{LC_zetai}
\mathscr{L}^C \zeta_i (y) \ = \  w_i(y) \quad \text{for all $y \in (\mathbb{T}^{\m}_{\ell} \times \mathbb{R}^{3 \m}) \setminus B_i$} \;.
\end{equation}
This is because when the deterministic flow is at a boundary point at time $t$ with either $\zeta_i = -\ell/2$ and $w_i>0$, or $\zeta_i=\ell/2$ and $w_i<0$, there exists a time interval $[t,t+h)$ such that the trajectory $s \mapsto \zeta^i_s$ is streatly increasing, (respectively, strictly decreasing) on $[t,t+h)$, and hence, there exists an integer $k$ such that $\zeta^i_s = z^i_s+k \ell$ for all $s \in [t,t+h)$.  Moreover, \begin{equation} \label{LC_wi}
\mathscr{L}^C w_i (y) \ = \ \nabla_i U(\tilde x) -  \nabla_i U(x)  \quad \text{for all $y \in \mathbb{T}^{\m}_{\ell} \times \mathbb{R}^{3 \m}$} \;.
\end{equation}
Since $r_i = \sqrt{|\zeta_i|^2+\alpha^{-2} |\zeta_i + \gamma^{-1} w_i|^2}$, $r_i^2$ is a smooth function of $(\zeta_i, w_i)$, and $r_i$ is a smooth function of $(\zeta_i, w_i)$ except at $(\zeta_i, w_i)=(0,0)$.  Thus, $\mathscr{L}^C(r_i^2)$ exists for all $y \in (\mathbb{T}^{\m}_{\ell} \times \mathbb{R}^{3 \m}) \setminus B_i$ and $\mathscr{L}^C r_i$ exists for all $y \in (\mathbb{T}^{\m}_{\ell} \times \mathbb{R}^{3 \m}) \setminus (Z_i \cup B_i)$.  

Expanding on this point, by \eqref{LC_zetai} and \eqref{LC_wi}, for all $y \in (\mathbb{T}^{\m}_{\ell} \times \mathbb{R}^{3 \m}) \setminus B_i$,  \begin{align}
& \mathscr{L}^C (r_i^2) \, = \, 2 \Big( \zeta_i w_i + \alpha^{-2} q_i w_i + \frac{1}{\gamma \alpha^{2}} q_i (\nabla_i U(\tilde x) - \nabla_i U( x) ) \Big) \nonumber \\
\, &= \, 2 \gamma  \Big( - |\zeta_i|^2 + \frac{|q_i|^2}{\alpha^2} + (1-\alpha^{-2}) \zeta_i q_i + \frac{1}{\gamma^2 \alpha^{2}} q_i (\nabla_i U(\tilde x) - \nabla_i U( x) ) \Big) \nonumber \\
\, &\le \, 2 \gamma  \Big(- |\zeta_i|^2 + \frac{|q_i|^2}{\alpha^2} + \big( 1-\alpha^{-2}+ \frac{L}{\gamma^2 \alpha^{2}} \big) |\zeta_i| |q_i| + \frac{J |q_i|}{\gamma^2 \alpha^2} \sum_{k \ne i} |\zeta_k| \Big) \nonumber \\
\, &\le \, 2 \gamma  \Big(- |\zeta_i|^2 + \frac{|q_i|^2}{\alpha^2} + 2 \frac{\alpha^2 -1}{\alpha^2} |\zeta_i| |q_i| + \frac{J |q_i|}{\gamma^2 \alpha^2} \sum_{k \ne i} |\zeta_k| \Big)  \label{LC_ri2} 
\end{align} where, in turn, we eliminated $w_i$ using $w_i = \gamma (q_i - \zeta_i)$, used \eqref{ieq:LJ} to bound $|\nabla_i U(\tilde x) - \nabla_i U( x)|$, and used $L \gamma^{-2} = \alpha^2 - 1$ which follows from \eqref{eq::gamma} and \eqref{eq::alpha}.   For all $y \in (\mathbb{T}^{\m}_{\ell} \times \mathbb{R}^{3 \m}) \setminus (B_i \cup Z_i)$, the chain rule and \eqref{LC_ri2} imply \begin{align}
   \mathscr{L}^C r_i  \, &\le \, \frac{\gamma}{r_i} \Big(- |\zeta_i|^2 + \frac{|q_i|^2}{\alpha^2} + 2\frac{\alpha^2 -1}{\alpha^2} |\zeta_i| |q_i| + \frac{J |q_i|}{\gamma^2 \alpha^2} \sum_{k \ne i} |\zeta_k| \Big) \label{eq:LC_ri} \\  \,  &\le \,  \gamma \alpha r_i + \frac{J}{\gamma \alpha} \sum_{k \ne i} \min(r_k, \ell/2)
        \label{ieq:LC_ri} 
        \end{align}
        where we used $|q_i|/\alpha \le r_i$, $|\zeta_k| \le \min(r_k, \ell/2)$, and the inequality 
        \begin{align*}
- |\zeta_i|^2 + \frac{|q_i|^2}{\alpha^2} + 2\frac{\alpha^2 -1}{\alpha^2} |\zeta_i| |q_i|  \le \max\big( \frac{\alpha^2 - 1}{\delta \alpha^2} - 1 ,1 + \delta (\alpha^2-1) \big) r_i^2 \le \alpha r_i^2  \;,
\end{align*}
where $\delta>0$ satisfies $(\alpha^2-1)/(\delta \alpha^2) - 1 = 1 + \delta (\alpha^2-1) = \sqrt{\alpha^{-2} - 1 + \alpha^2} \le \alpha$, since $\alpha \ge 1$.
\begin{remark} \label{weak_LC_ri}
For $y \in (\mathbb{T}^{\m}_{\ell} \times \mathbb{R}^{3 \m}) \setminus B_i$, \eqref{ieq:LC_ri} holds in a weak sense.  To see this, approximate $r_i(y)$ by $r_{i,\epsilon}(y) \, := \, \varphi_{\epsilon}( ( r_i(y) )^2 )$ where $\epsilon>0$ is a small parameter and $\varphi_{\epsilon}$ is a $C^1$ function defined by $\varphi_{\epsilon}(x) = \sqrt{x}$ for $x \ge \epsilon^2$ and $\varphi_{\epsilon}(x) = \epsilon/2+x/(2 \epsilon)$ for $x \le \epsilon^2$.  By the standard chain rule, \[
 \mathscr{L}^C r_{i,\epsilon} \, = \,  \varphi_{\epsilon}'((r_i)^2) \mathscr{L}^C (r_i)^2 \, = \, \frac{1}{2 \max(r_i, \epsilon)}   \mathscr{L}^C (r_i)^2  \;,
\] and thus for any $y_0 \in (\mathbb{T}^{\m}_{\ell} \times \mathbb{R}^{3 \m})  \setminus B_i$ and $t \ge 0$ sufficiently small, \begin{align*}
r_{i,\epsilon}(y_t) - r_{i,\epsilon}(y_0)  =  \int\limits_0^t \frac{ \mathscr{L}^C (r_i)^2 (y_s)}{2 \max(r_s^i, \epsilon)} ds   \le   \int\limits_0^t \big( \gamma \alpha r^i_s + \frac{J}{\gamma \alpha} \sum_{k \ne i} \min(r^k_s,\ell/2) \big) ds . 
\end{align*} As $\epsilon \downarrow 0$, $r_{i,\epsilon}(y_t) \downarrow r_t^i$ and thus we obtain the same bound for $r_t^i$, i.e., \eqref{ieq:LC_ri} holds in a weak sense.  
\end{remark}

\begin{lemma} \label{LC_tor_lemma}
Suppose that $\lambda>0$ satisfies condition~\eqref{eq:lambcond_tor}.
For any $i \in \{1, \dots, \m\}$, for any initial condition $y \in \mathbb{T}^{\m}_{\ell} \times \mathbb{R}^{3 \m}$, for any $t>0$, and for any $s \in [0,t]$, 
 \begin{align} \label{LC_f_tor}
f(r^i_t) - f(r^i_s) \  &\le \ \int_s^t \mathrm{g}_i( y_u ) d u \;,~\text{where $\mathrm{g}_i: \mathbb{T}^{\m}_{\ell} \times \mathbb{R}^{3 \m} \to \mathbb{R}$ is given by} \\
    \mathrm{g}_i(y) \ &:= \ \begin{cases} \
    0 & \text{for $r_i > \mathcal{R}$} \;, \\
    f'(r_i(y)) \mathscr{L}^C r_i(y) & \text{for $0<r_i \le \mathcal{R}$} \;, \\
   \  J /( \gamma \alpha)  \sum_{k \ne i} \min(r_k(y), \ell/2) & \text{for $r_i=0$ }  \;.
    \end{cases} \label{gi}
\end{align}
\end{lemma}

Lemma~\ref{LC_tor_lemma} states that in the weak sense $\mathcal{L}^C (f \circ r_i) \le \mathrm{g}_i$.  We  know this holds with equality for $y \in \mathbb{T}^{\m}_{\ell} \times \mathbb{R}^{3 \m}$ such that $|\zeta_i|<\ell/2$ and $y \notin Z_i$.  This lemma extends this equality to an inequality that is valid globally.

\begin{proof}
It suffices to prove \eqref{LC_f_tor} for $s=0$, and use the flow property to write $f(r^i_t) - f(r^i_s) = f \circ r_i( \phi^C_t(y) ) - f \circ r_i(\phi^C_s(y)) = f \circ r_i(\phi^C_{t-s}(y_s)) -f \circ r_i (y_s)$, i.e., start the underlying flow with initial condition $y_s$ instead of $y$.  From now on, we assume w.l.o.g.~that $s=0$. Fix an $\epsilon>0$ and introduce the function $\mathrm{g}_{i,\epsilon}: \mathbb{T}^{\m}_{\ell} \times \mathbb{R}^{3 \m} \to \mathbb{R}$ defined by \begin{equation*}
    \mathrm{g}_{i,\epsilon}(y) \, := \, \begin{cases} \
    \mathrm{g}_i(y) & \text{$r_i > \epsilon$,}  \\
   \ \gamma \alpha r_i(y) + J /( \gamma \alpha)  \sum_{k \ne i} \min(r_k(y), \ell/2) & \text{$r_i \le \epsilon$.} 
    \end{cases}
\end{equation*}
Below we prove \eqref{LC_f_tor} holds with $\mathrm{g}_i$ replaced with $\mathrm{g}_{i,\epsilon}$, i.e., \begin{equation}
    \label{LC_f_tor_2}
f(r^i_t) - f(r^i_0) \le \int_0^t \mathrm{g}_{i,\epsilon}( y_u ) d u \;.
\end{equation} Then \eqref{LC_f_tor}  follows since as $\epsilon \downarrow 0$ we have $\mathrm{g}_{i,\epsilon} \downarrow \mathrm{g}_{i}$. 
Define  \[
\tau \, := \, \sup\left\{ u \ge 0 : \text{\eqref{LC_f_tor_2} holds for all $t \in [0,u]$} \right\} \;.  
\] We will prove $\tau=\infty$ by contradiction.  Hence suppose $\tau < \infty$.  By Lemma~\ref{lem:cadlag} and monotonicity of $f$, \eqref{LC_f_tor_2} holds for all $t \in [0, \tau]$ with $\tau$ included.  Indeed, by definition of $\tau$, \eqref{LC_f_tor_2} holds for $t<\tau$.  Moreover, by Lemma~\ref{lem:cadlag}, \[
f(r^i_{\tau}) - f(r_0^i) \le \lim_{t \uparrow \tau} f(r^i_t) - f(r_0^i) \;.
\] Thus, since the r.h.s.~of \eqref{LC_f_tor_2} is continuous in $t$, this bound extends from $t<\tau$ to $t=\tau$.  

\medskip

Now we distinguish several cases depending on the size of $r_{\tau}^i$.
\paragraph{Case (i): $r_{\tau}^i > \mathcal{R}$} 
Note, first, that this case includes $r_{\tau}^i \in B_i$ by Remark~\ref{Bi}.  In this case, by right continuity of $t \mapsto r^i_t$, there exists $h>0$ such that for all $t \in [\tau, \tau+h]$ we have $r^i_t>\mathcal{R}$, and hence, $f'(r^i_t)=0$ and $f(r^i_t) = f(\mathcal{R}) = f(r^i_{\tau})$.  Inserting these results into \eqref{LC_f_tor_2} gives $f(r^i_t) - f(r^i_0) \le \int_0^{\tau} \mathrm{g}_{i,\epsilon}( y_u ) d u = \int_0^{t} \mathrm{g}_{i,\epsilon}( y_u ) d u$. Thus, \eqref{LC_f_tor_2} holds for all $t \in [\tau, \tau + h]$, which contradicts the definition of $\tau$.  

\paragraph{Case (ii): $\epsilon < r_{\tau}^i \le \mathcal{R}$ and $|\zeta_{\tau}^i| < \ell/2$}
In this case, there exists $h>0$ such that  $|\zeta^i_t|<\ell/2$ and $r^i_t > \epsilon$ for all $t \in [\tau-h, \tau+h]$.  Therefore, $t \mapsto r^i_t$ is smooth on this interval, and thus for  $t \in [\tau-h, \tau+h]$,  since $f$ is Lipschitz continuous\footnote{From \eqref{eq:f}, note that $f$ is a composition of two Lipschitz functions, and hence,  \[
|f(r) - f(s)| = a^{-1} |e^{-a (r \wedge \mathcal{R})} - e^{-a (s \wedge \mathcal{R})}| \le a | r \wedge \mathcal{R} - s \wedge \mathcal{R}| \le a | r - s|  \quad \text{for all $r, s \in [0,\infty)$} \;.
\]}, it is also absolutely continuous, and therefore, $f(r^i_t) - f(r^i_{\tau}) = \int_{\tau}^t f'(r^i_u) \, \mathscr{L}^C r_i (y_u) \, du$.  Thus, \eqref{LC_f_tor_2} holds for all $t \in [\tau, \tau + h]$, which contradicts the definition of $\tau$.  
Here we used that every Lipschitz continuous function is absolutely continuous. 


\paragraph{Case (iii): $\epsilon < r_{\tau}^i \le \mathcal{R}$ and $|\zeta_{\tau}^i|=\ell/2$} This case can be treated similarly to case (ii).  Note, first, that $\ell/2 \le r_{\tau}^i \le \mathcal{R}$, and hence, $r_{\tau}^i \notin B_i$ by Remark~\ref{Bi}.  Suppose, for example, that $\zeta^i_{\tau}=-\ell/2$ and $w^i_{\tau}>0$.  Then $z^i_t$ is strictly increasing for $t$ near $\tau$.  Therefore, for $t \in [\tau, \tau+h]$ with $h$ sufficiently small, $\zeta^i_{t}$ is strictly increasing, $\zeta^i_{t} \in (-\ell/2,0)$ and $r^i_{t}>\epsilon$. In particular, for $t \in [\tau, \tau+h]$, $\zeta^i_t = z^i_t + k \ell$ for a fixed integer $k$, and thus, $r^i_t$ is a smooth function for all $t \in [\tau, \tau + h]$.   (The only difference to case (ii) is that now,  smoothness of $r^i_t$ holds only for $t \in [\tau, \tau + h]$, and not for $t \in [\tau - h, \tau +h]$.)  Now we can argue completely analogously to case (ii) to conclude that \eqref{LC_f_tor_2} holds for all $t \in [\tau, \tau + h]$, which contradicts the definition of $\tau$.

\paragraph{Case (iv): $0 < r_{\tau}^i \le \epsilon$} 
In this case, there exists $h>0$, such that $r^i_t$ is smooth for $t \in [\tau - h, \tau+h]$, and similar to case (ii), $f(r^i_t) - f(r^i_{\tau}) = \int_{\tau}^t f'(r^i_u) \, \mathscr{L}^C r_i (y_u) \, du$.  However, we only have $r^i_{t}>0$.  This motivates inserting $0 \le f' \le 1$ into the bound in \eqref{ieq:LC_ri} to obtain that for $t \in [\tau - h, \tau+h]$ \[
f'(r^i_t) \mathscr{L}^C r_i (y_t) \le \gamma \alpha r_t^i + J/(\gamma \alpha) \sum_{k \ne i} \min(r_t^k, \ell/2) \;. 
\]  Thus, \eqref{LC_f_tor_2} holds for all $t \in [\tau, \tau + h]$, which contradicts the definition of $\tau$. 

\paragraph{Case (v): $r_{\tau}^i=0$} 
In this case $r^i_t$ is not smooth at $\tau$, but as noted in Remark~\ref{weak_LC_ri} the bound in \eqref{ieq:LC_ri} can still be applied in a weak sense.  In particular, for $t \in [\tau, \tau+h]$ with $h$ sufficiently small we have $r^i_t < \epsilon$ and \[
f(r^i_t) - f(r^i_{\tau}) \, \le \,
\int_{\tau}^t \big( \gamma \alpha r^i_u + J/(\gamma \alpha) \sum_{k \ne i} \min(r^k_u,\ell/2) \big) du \;.
\]  Hence, we see again that for $h$ sufficiently small, \eqref{LC_f_tor_2} extends to $t \in [\tau, \tau+h]$, but this contradicts the definition of $\tau$.
\end{proof}

\subsection{Combined bounds for generator of Andersen dynamics on $\mathbb{T}^{\m}_{\ell}$ acting on metric}

The following theorem uses Lemmas~\ref{AC_tor_lemma} and~\ref{LC_tor_lemma} to bound, in the weak sense, the generator of Andersen dynamics acting on $\rho$.

\begin{theorem} \label{GCrho_tor}
Suppose that $\lambda>0$ satisfies condition~\eqref{eq:lambcond_tor}, and \eqref{eq:Jbound} holds. Then for every  $y \in \mathbb{T}^{\m}_{\ell} \times \mathbb{R}^{3 \m}$ \begin{align*}
& \sum_i \left( \mathrm{g}_{i} + \mathcal{A}^C_{\gamma} (f \circ r_i) \right)   \le  - \cA   \sum_i f(r_i),~
\text{where $\cA $ is defined in \eqref{eq:rate_tor}.}
\end{align*}
\end{theorem}

Theorem~\ref{GCrho_tor} states that in the weak sense \[
\frac{d}{dt} \rho(Y_t) = \mathcal{G}^C_{\gamma} \rho(Y_t) \le  - \cA  \rho(y) \;,
\] for every initial condition $y \in \mathbb{T}^{\m}_{\ell} \times \mathbb{R}^{3 \m}$.

\begin{proof}
First we combine the bounds from Lemmas~\ref{AC_tor_lemma} and~\ref{LC_tor_lemma} to obtain a global component-wise bound, and then sum over these bounds to obtain an overall global bound on the generator of Andersen dynamics acting on $\rho$.

\paragraph{Bound for $r_i>\mathcal{R}$}
Applying Lemmas~\ref{AC_tor_lemma} and~\ref{LC_tor_lemma} in this case gives \begin{equation} \label{ri_large}
\mathrm{g}_{i} + \mathcal{A}^C_{\gamma} (f \circ r_i) \, \le \,  - \tilde{c}_0 f(r_i) \;,~\text{where $\tilde{c}_0 \, := \, \frac{9}{25} \frac{\lambda}{\m} \exp\left(- \beta^{1/2} \frac{ \lambda}{\m} \frac{\ell}{2} \right)$} \;.      
\end{equation}

\paragraph{Bound for $0<r_i \le \mathcal{R}$}
By Lemma~\ref{LC_tor_lemma} and \eqref{eq:LC_ri}, \begin{align*}
&    \mathrm{g}_{i}  = f'(r_i) \mathscr{L}^C r_i \le -\gamma  \frac{f'(r_i)}{r_i}  \big( |\zeta_i|^2 - \frac{|q_i|^2}{\alpha^2} + 2\frac{1-\alpha^2}{\alpha^2} |\zeta_i| |q_i| - \frac{J |q_i|}{\gamma^2 \alpha^2} \sum_{k \ne i} |\zeta_k|  \big) 
\end{align*} 
Combining this bound with Lemma~\ref{AC_tor_lemma}, we obtain \begin{align*}
 \mathrm{g}_{i} + \mathcal{A}^C_{\gamma} (f \circ r_i)  \, \le \,   - \gamma \frac{f'(r_i)}{r_i} \Big( \frac{1}{20} r_i^2 + Q(|\zeta_i|, |q_i|) -  \frac{J}{\gamma^{2} \alpha^{2}} \sum_{k \ne i} |\zeta_k| |q_i| \Big)  \;.
\end{align*}
Here we have introduced the quadratic form \begin{align*}
& Q(|\zeta_i|, |q_i|) \, := \, \frac{11}{20} |\zeta_i|^2 + \left( \frac{3}{10} \frac{\lambda}{\gamma \m} - \frac{5}{4} \right) \frac{|q_i|^2}{\alpha^2} - 2 \frac{\alpha^2 -1}{\alpha^2}  |\zeta_i| |q_i| \\
& \quad \, = \, \frac{11}{20} |\zeta_i|^2 + \left( \frac{3}{10} \frac{\beta^{1/2} \lambda \mathcal{R}}{\m} - \frac{5}{4} \right) \frac{|q_i|^2}{1+\beta L \mathcal{R}^2} - 2 \frac{\beta L \mathcal{R}^2}{1+\beta L \mathcal{R}^2} |\zeta_i| |q_i| 
\end{align*}
where in the last expression we eliminated $\gamma$ and $\alpha$ using \eqref{eq::gamma} and \eqref{eq::alpha}.  This quadratic form is nonnegative provided that \[
\left( \frac{\beta L \mathcal{R}^2}{1+\beta L \mathcal{R}^2} \right)^2 \le \frac{33}{200} \left( \frac{\beta^{1/2} \lambda \mathcal{R}}{\m} - \frac{25}{6}  \right) \frac{1}{1+\beta L \mathcal{R}^2} \;.
\] A sufficient condition for this condition to hold is \begin{equation} \label{A1}
   \frac{\beta^{1/2} \lambda \mathcal{R}}{\m}   
    \ge \frac{25}{6} + \frac{200}{33} \beta L \mathcal{R}^2 \;.
\end{equation} Moreover, since (as noted in Remark~\ref{Bi}) $\mathcal{R} \le (4/3) (\ell/2)$, condition~\eqref{eq:lambcond_tor} implies condition~\eqref{A1} because \[
\frac{\beta^{1/2} \lambda \mathcal{R}}{\m} \ge 
\frac{\beta^{1/2} \lambda}{\m} \frac{\ell}{2} \ge 
\frac{25}{6} + 11 \frac{9}{16} \beta L \mathcal{R}^2 \ge
\frac{25}{6} + \frac{200}{33} \beta L \mathcal{R}^2 \;.
\]
Thus, under condition~\eqref{eq:lambcond_tor}, we obtain \begin{align}
& \mathrm{g}_{i} + \mathcal{A}^C_{\gamma} (f \circ r_i) \, \le \,
 - \gamma \frac{f'(r_i)}{r_i} \Big( \frac{1}{20} r_i^2 -  \frac{J}{\gamma^{2} \alpha^{2}} \sum_{k \ne i} |\zeta_k| |q_i| \Big)  \nonumber \\
 &\le -\frac{\gamma}{20} r_i f'(r_i) + \frac{J}{\gamma \alpha} \sum_{k \ne i} |\zeta_k|  = 
 -\frac{\gamma}{20} \frac{a r_i}{e^{a r_i} - 1} f(r_i)  + \frac{J}{\gamma \alpha} \sum_{k \ne i} |\zeta_k| 
 \nonumber \\
 &\le  -\tilde{\tilde{c}}_0 f(r_i) + \frac{J}{\gamma \alpha} \sum_{k \ne i} \min(r_k, \ell/2) \;,~\text{where $\tilde{\tilde{c}}_0 \, := \, \frac{\gamma}{20} \frac{a \mathcal{R}}{e^{a \mathcal{R}} - 1}$} \;,  \nonumber 
\end{align}
where we used monotonicity of $\mathsf{x}/(e^{\mathsf{x}} -1 )$ for $\mathsf{x}>0$. By \eqref{eq::R} and \eqref{eq::a},  \[
c_0:=\frac{\lambda}{55 \m} e^{-\beta^{1/2} \frac{\lambda}{\m} \frac{\ell}{2} } \le \tilde{\tilde{c}}_0 = \frac{\lambda}{20 \m} \frac{e^{-\beta^{1/2} \frac{\lambda}{\m} \frac{\ell}{2} }}{e  - e^{-\beta^{1/2} \frac{\lambda}{\m} \frac{\ell}{2} }} \le 
\frac{\lambda}{20 \m} e^{-\beta^{1/2} \frac{\lambda}{\m} \frac{\ell}{2} } < \tilde{c}_0 \;. 
\] 
Hence, the following bound holds for $r_i>0$, \begin{equation} \label{ri_small}
    \mathrm{g}_{i} + \mathcal{A}^C_{\gamma} (f \circ r_i) \, \le \, -c_0 f(r_i) + J /( \gamma \alpha) \sum_{k \ne i} \min(r_k, \ell/2) \;.
\end{equation}

\paragraph{Bound for $r_i=0$} 
Applying Lemmas~\ref{AC_tor_lemma} and~\ref{LC_tor_lemma} in this case gives \begin{equation} \label{ri_zero}
\mathrm{g}_{i} + \mathcal{A}^C_{\gamma} (f \circ r_i) \, \le \, J /( \gamma \alpha)  \sum_{k \ne i} \min(r_k, \ell/2) \;.    
\end{equation}
Thus, the component-wise bound in \eqref{ri_small} holds globally under condition~\eqref{eq:lambcond_tor}.

\paragraph{Overall global bound} Summing over the component-wise bounds in \eqref{ri_small}, \begin{align*}
  \sum_i   \Big( \mathrm{g}_{i} &+ \mathcal{A}^C_{\gamma} (f \circ r_i) \Big) \, \le \,  -c_0 \sum_i f(r_i) + J /( \gamma \alpha) (\m - 1) \sum_i \min(r_i, \ell/2) 
  \;, \\
  \, &\le \, \Big( -c_0 + \frac{J}{\gamma \alpha} (\m - 1) \frac{\ell/2}{f(\ell/2)} \Big)  \sum_i f(r_i)
  \;, \\
    \, &\le \, \Big( -c_0 + \frac{J}{\gamma \alpha} (\m - 1) \frac{\beta^{1/2} \frac{\lambda}{\m} \frac{\ell}{2}}{1-e^{-\beta^{1/2} \frac{\lambda}{\m} \frac{\ell}{2}}} \Big)  \sum_i f(r_i)  \;, \\
    \, &\le \, 
- \Big( c_0 -  \frac{51}{50} \frac{J}{\gamma \alpha} (\m - 1) \beta^{1/2} \frac{\lambda}{\m} \frac{\ell}{2} \Big)     
    \, \le \,-\cA   \sum_i f(r_i) \;, ~\text{as required.}
\end{align*}
Here, in turn, we used that $\min(r_i, \ell/2) \le (\ell/2) (f(r_i)/f(\ell/2))$; condition~\eqref{A1}, which implies $\beta^{1/2} (\lambda/\m) (\ell/2) \ge 25/6 > 4$ and hence $e^{-\beta^{1/2} (\lambda/\m) (\ell/2)} < e^{-4}$ and by \eqref{eq::R}, $\mathcal R\le \ell /2+\ell /8<\ell$. The last inequality then
follows from \eqref{eq:Jbound} which implies that
\begin{equation*}
   {J}\ \le\ \frac{1}{150(\m -1)}\frac{2}{\ell}\max\left({L}^{1/2},\beta^{-1/2}\mathcal R^{-1}\right)\beta^{-1/2}
    \exp\left(- \beta^{1/2} \frac{ \lambda}{\m} \frac{\ell}{2}\right),
\end{equation*}
and since $\gamma \alpha = \beta^{-1/2} \mathcal{R}^{-1}\sqrt{1+\beta L \mathcal{R}^2} \ge \max(L^{1/2}, \beta^{-1/2} \mathcal{R}^{-1})$.
\end{proof}

\subsection{Proof of main contraction result for Andersen dynamics on $\mathbb{T}^{\m}_{\ell}$}

\label{SEC:TOR_PROOF}

In this part, we show that $M_t \, = \, e^{\cA \, t} \sum_{i} f(r_i(Y_t))$ is a nonnegative supermartingale.  To this end, we  develop a Dynkin-like inequality for $e^{\cA \, t} f(r_i(Y_t))$, and as an intermediate step, we prove an analogous result for the deterministic solution $y_t$ of \eqref{odes:tor} starting at $y_0 = y$.

\begin{lemma} \label{lemma_D0}
Let $r^i_t:=r_i(y_t)$,  fix $c \in (0, \infty)$ and suppose that
$\mathrm{g}_i : \mathbb{T}^{\m}_{\ell} \times \mathbb{R}^{3 \m} \to \mathbb{R}$ satisfies $f(r^i_t) - f(r^i_s) \  \le \ \int_s^t \mathrm{g}_i( y_u ) d u$.  For all $t>0$, $0 \le s \le t$, $y \in \mathbb{T}^{\m}_{\ell} \times \mathbb{R}^{3 \m}$ with $y_0=y$, and $i\in \{1, \dots, \m \}$,\[
 e^{c \, t} f(r^i_t) - e^{c \, s} f(r^i_s) \,  \le  \, \int\limits_s^t e^{c \, u} \Big( \mathrm{g}_i + c \, f \circ r_i \Big)(y_u) du  \;.
\]   
\end{lemma}

Formally, Lemma~\ref{lemma_D0} follows by the chain rule, but since $f(r^i_t)$ is not differentiable, we give a direct proof.

\begin{proof} Fix a sequence $(\Pi_n)_{n \in \mathbb{N}}$ of partitions of $[s,t]$ such that $\Pi_n \subseteq \Pi_{n+1}$ such that the mesh size $\| \Pi_n \| \to 0$ as $n \to \infty$.  Let $v \, := \, \min\{ r \in \Pi_n : r > u \}$ denote the next partition point after $u$.  Let $A_t := e^{c t}$, $F_t := f(r^i_t)$, and $\sum_u:=\sum_{\substack{u \in \Pi_n \\  u<t}}$. Then $A_t F_t - A_s F_s = \sum_{u} \big( A_v F_v - A_u F_u \big)$, and hence, \begin{align}
A_t F_t - A_s F_s =  \rn{1} + \rn{2} + \rn{3} ~~ \text{where} ~
\begin{cases} 
\rn{1} := \sum_u F_u (A_v-A_u),  \\
\rn{2} := \sum_u A_u (F_v-F_u), \\
\rn{3} := \sum_u (A_v-A_u) (F_v-F_u) .
\end{cases}
\end{align} 
As $n \to \infty$, we have: \[
\rn{1} = \int_{[s,t]} \sum_u F_u \mathbbm{1}_{[u,v]}(r) d A_r \to \int_{[s,t]} F_{r-} d A_r = \int_{[s,t]} F_{r} \dot{A}_r dr \;,
\]
by dominated convergence and continuity of $\dot{A}_r$; \[
\rn{2} \le \sum_u A_u \int_u^v \mathrm{g}_i(y_r) dr \to \int_{[s,t]} A_r \mathrm{g}_i(y_r) dr \;,
\]
by continuity of $A_r$ and dominated convergence; and, 
\[
\rn{3} \le c e^{c t} \sum_u (v-u) \int_u^{v} \mathrm{g}_i(y_r) dr \to 0 \;.
\] Hence, $A_t F_t - A_s F_s \le \liminf_{n \to \infty} (\rn{1}+\rn{2}+\rn{3}) \le \int_{[s,t]} (A_r \mathrm{g}_i(y_r) + \dot{A}_r F_r ) dr$, as required.
\end{proof}

The next lemma applies Lemma~\ref{lemma_D0} to obtain a Dynkin-like inequality for $e^{\cA \, t} f \circ r_i(Y_t)$ where $Y_t$ is the coupling process.

\begin{lemma} \label{lemma_D1}
Suppose that $\lambda>0$ satisfies condition~\eqref{eq:lambcond_tor}.
Let $R^i_t:= r_i(Y_t)$.  
There exists $C>0$ such that for all $t \in [0,1]$, $y \in \mathbb{T}^{\m}_{\ell} \times \mathbb{R}^{3 \m}$ with $Y_0=y$, and $i\in \{1, \dots, \m \}$,  \[
\E \Big( e^{\cA \, t} f(R^i_t) - f(R^i_0) \Big) \le e^{-\lambda t} \int\limits_0^t e^{\cA \, s} \Big(\mathrm{g}_i + \mathcal{A}^C_{\gamma} (f \circ r_i) + \cA \, f \circ r_i \Big)(y_s) ds + C t^2 
\]
where $y_s$ denotes the deterministic solution to \eqref{odes:tor} with the same initial condition $y_0 = y$. 
\end{lemma}

\begin{proof}
Recall from Definition~\ref{D:coupling}, $N_t$ represents the number of velocity randomizations that have occurred over $[0,t]$, and $T_1$ is the first jump time.
Introduce the decomposition $e^{c t} f(R_t^i) - f(R_0^i) =  \rn{1} + \rn{2} + \rn{3} $ where \begin{align*}
\begin{cases} 
\rn{1} \ := \ \big( e^{c t} f(R_t^i) - f(R_t^0) \big) \mathbbm{1}_{\{N_t=0\}} \;,  \\
\rn{2} \ := \ e^{c T_1} \big( f(R_{T_1}^i) - f(R^i_{T_1-}) \big)  \mathbbm{1}_{\{N_t \ge 1\}} \;,  \quad \text{and} \\
\rn{3} \ := \ \Big( e^{c t} f(R_t^i) - e^{c T_1} f(R_{T_1}^i)  + e^{c T_1} f(R^i_{T_1-}) -   f(R_t^0)  \Big) \mathbbm{1}_{\{N_t \ge 1\}} \;.
\end{cases}
\end{align*} 
We now bound the expectations of $\rn{1}$, $\rn{2}$, and $\rn{3}$.  

On $N_t = 0$, we have $R_s^i = r_s^i$ for all $s \le t$, where recall $r_s^i = r_i(y_s)$ denotes the corresponding distance function for the deterministic solution.  Hence, for all $y \in \mathbb{T}^{\m}_{\ell} \times \mathbb{R}^{3 \m}$, by Lemma~\ref{lemma_D0}, \begin{align}
\E( \rn{1} ) &= \Pr(N_t=0) (e^{\cA \, t} f(r^i_t) - f(r^i_0) ) \nonumber \\
&\le e^{-\lambda t} \int_0^t \big( e^{\cA \, s} \mathrm{g}_i(y_s) + \cA e^{\cA \, s} f(r^i_s) \big)  ds \;. \label{lemma_D1:1}
\end{align}

To bound $\rn{2}$, note that the event $\{ N_t \ge 1 \}$ is equivalent to the event $\{ T_1 \le t \}$, and that $Y_s = y_s$ for $s<T_1$, and hence, $Y_{T_1-}=y_{T_1-}=y_{T_1}$.  Thus, we can write \begin{align*}
\rn{2} &= e^{\cA \, T_1} \Big( f \circ r_i (\mathcal{S}^C(I_1, \xi_1, \mathcal{U}_1) y_{T_1} ) - f \circ r_i(y_{T_1}) \Big)  \mathbbm{1}_{\{T_1 \le t\}} \\
&= e^{\cA \, T_1} \Big( f \circ r_i (\mathcal{S}^C(i, \xi_1, \mathcal{U}_1) y_{T_1} ) - f \circ r_i(y_{T_1}) \Big)  \mathbbm{1}_{\{T_1 \le t, I_1 = i \}}
\end{align*} Since $T_1$, $\xi_1$, $\mathcal{U}_1$ and $I_1$ are independent, the conditional expectation of $\rn{2}$ given $T_1$ is given by \begin{align}
 \E(\rn{2} \mid T_1 = s) &= \frac{1}{\m} \E \Big( f \circ r_i (\mathcal{S}^C(i, \xi_1, \mathcal{U}_1) y_{s} ) - f \circ r_i(y_{s}) \Big) e^{\cA \, s} \mathbbm{1}_{\{s \le t \}} \nonumber \\
& = \frac{1}{\lambda} \mathcal{A}^C_{\gamma} ( f \circ r_i) (y_s) e^{\cA \, s} \mathbbm{1}_{\{s \le t \}} \;, \quad \text{and thus,} \nonumber  \\
    \E( \rn{2} ) = \int_0^{\infty} & \E(\rn{2} \mid T_1 = s)  \lambda e^{-\lambda s} ds = \int_0^t e^{(\cA - \lambda) s} \mathcal{A}^C_{\gamma} ( f \circ r_i) (y_s) ds \;.
    \label{lemma_D1:2}
\end{align}

Now we show that $\E(\rn{3})$ is of order $O(t^2)$ for small $t$.  For this purpose, we introduce the decomposition $\rn{3} = \rn{3}_a + \rn{3}_b + \rn{3}_c$ where 
\begin{align*}
\begin{cases} 
\rn{3}_a \ := \ \big( e^{c T_1} f(R^i_{T_1-}) -   f(R_t^0)  \big) \mathbbm{1}_{\{N_t \ge 1\}} \;,  \\
\rn{3}_b \ := \  \big( e^{c t} f(R_t^i) - e^{c T_1} f(R_{T_1}^i) \big)  \mathbbm{1}_{\{N_t = 1\}} \;, \quad \text{and} \\
\rn{3}_c \ := \ \Big( e^{c t} f(R_t^i) - e^{c T_1} f(R_{T_1}^i)   \Big) \mathbbm{1}_{\{N_t \ge 2\}} \;.
\end{cases}
\end{align*}

To bound $\E(\rn{3}_a)$,  note from \eqref{ieq:LC_ri} that $\mathrm{g}_i$ in \eqref{gi} is globally bounded by a constant $C_{\mathrm{g}}$ and that $R^i_s = r^i_s$ for $s<T_1$. Thus, by Lemma~\ref{lemma_D0}, there exists a constant $C_a>0$ such that \begin{align*}
    \E( \rn{3}_a ) &= \E( e^{ \cA \, T_1} f(r^i_{T_1}) - f(r^i_0);~T_1 \le t ) \\
    &\le t \big( e^{\cA \, t} C_{\mathrm{g}} + \cA e^{\cA \, t} f(\mathcal{R}) \big) (1-e^{-\lambda t}) \le C_a t^2 ~~ \text{for all $t \le 1$}  \;. 
\end{align*}
Her we have used that by Lemma~\ref{lemma_D0} and since $\mathrm{g}_i \le C_{\mathrm{g}}$,\begin{align*}
   e^{\cA \, s} f(r_s^i) - f(r_0^i) &\le \int_0^s \big( e^{\cA \, u }  C_{\mathrm{g}} + \cA e^{\cA \, u} f(r^i_u) \big) du \;,  \\
&  \le s \big( e^{\cA \, s} C_{\mathrm{g}} + \cA e^{\cA \, s} f(\mathcal{R} \big) ~~ \text{for all $s \ge 0$} \;. 
\end{align*}. A similar bound holds for $\E(\rn{3}_b)$, since on $\{N_t = 1\}$ $R^i_s = r^i(\phi^C_{s-T_1}(Y_{T_1}) =: \tilde r^i_{s-T_1}$ for all $s \in [T_1, t]$ where $\tilde r_u^i$ is the distance for the deterministic solution $\tilde y_u$ with initial condition $\tilde y_0 = Y_{T_1}$.  Hence, by Lemma~\ref{lemma_D0}, on $\{N_t = 1 \}$,\begin{align*}
e^{\cA \, t } r(R_t^i) &- e^{\cA \, T_1} f(R_{T_1}^i = e^{\cA \, T_1} \big( e^{\cA \, (t-T_1)} f(\tilde r_{t-T_1}^i) - f(\tilde{r}_0) \big) \\
&\le t (e^{\cA \;, t} C_{\mathrm{g}} + \cA e^{\cA \, t} f(\mathcal{R} \big) \;.
\end{align*} Thus, we obtain similarly as in $\E(\rn{3}_a)$, \begin{align*}
    \E( \rn{3}_b ) &= \E( e^{ \cA \, t} f(r^i_{T_1}) - e^{\cA \, T_1} f(r^i_0);~T_1 \le t ) \\
    &\le t \big( e^{\cA \, t} C_{\mathrm{g}} + \cA e^{\cA \, t} f(\mathcal{R}) \big) (1-e^{-\lambda t}) \le C_a t^2 ~~ \text{for all $t \le 1$}  \;. 
\end{align*} To bound $\E(\rn{3}_c)$, a rough bound suffices, 
\begin{align*}
    \E( \rn{3}_c ) &\le  \Pr( N_t \ge 2 ) e^{\cA \, t} f(R) \\
    &\le (1-e^{-\lambda t} - \lambda e^{-\lambda t}) e^{\cA \, t} f(R) 
     \le C_c  t^2 ~~ \text{for all $t \le 1$}  \;,
\end{align*}
with a finite constant $C_c>0$. In sum, we obtain 
\begin{equation}
    \E(\rn{3}) \le (2 C_a + C_c) t^2 ~~ \text{for all $t \le 1$} \;. \label{lemma_D1:3}
\end{equation}

Combining \eqref{lemma_D1:1}, \eqref{lemma_D1:2}, and \eqref{lemma_D1:3} we obtain for $t \le 1$: \begin{align*}
  & \E\big( e^{\cA \, t} f(R_t^i) - f(R_0^i) \big) \le e^{-\lambda t} \int_0^t e^{\cA \, s} \Big( \mathrm{g}_i + \mathcal{A}^C_{\gamma}(f \circ r_i) + \cA f \circ r_i \Big) (y_s) d s \\
  & \qquad \qquad \qquad  + \int_0^t e^{\cA \, s} (e^{-\lambda s} - e^{-\lambda t} ) \mathcal{A}^C_{\gamma}(f \circ r_i) (y_s) ds + (2 C_a + C_c) t^2 \\
  &\qquad \qquad \le e^{-\lambda t} \int_0^t e^{\cA \, s} \Big( \mathrm{g}_i + \mathcal{A}^C_{\gamma}(f \circ r_i) + \cA f \circ r_i \Big) (y_s) d s + C t^2 \;,
\end{align*}
with a finite constant $C$ depending only on the parameters $\lambda$, $\m$, etc.
In the last step, we used $|e^{-\lambda s} - e^{-\lambda t} | \le \lambda (t-s)$ and the crude bound $\mathcal{A}^C_{\gamma}(f \circ r_i) (y_s) \le \lambda f( \mathcal{R} )$ to obtain $\int_0^t e^{\cA \, s} (e^{-\lambda s} - e^{-\lambda t} ) \mathcal{A}^C_{\gamma}(f \circ r_i) (y_s) ds \lesssim t^2$. 
\end{proof}

Now we combine Lemma~\ref{lemma_D1} and Theorem~\ref{GCrho_tor} to prove that $M_t$ is a nonnegative supermartingale for every initial condition $Y_0 = y \in \mathbb{T}^{\m}_{\ell} \times \mathbb{R}^{3 \m}$.

\begin{proof}[Proof of Theorem~\ref{TOR}]
Let  $\rho_t := \sum_{i=1}^{\m} f \circ r_i( Y_t )$ with $Y_0 = y$ and let $\E_y$ denote expectation conditional on $Y_0 = y$; as the notation indicates, the underlying measure space depends on the initial point.  By Lemma~\ref{lemma_D1} and Theorem~\ref{GCrho_tor}, we obtain \[
 \E_y( M_t - M_0 ) \, \le \, \m \, C \, t^2  ~~ \text{for all $t \in [0,1]$ and $y \in \mathbb{T}^{\m}_{\ell} \times \mathbb{R}^{3 \m}$} \;. 
 \] Fix $h \in [0,1]$.  Then \[
 \E_y( M_t - M_0 ) \, \le \, \m \, C \, h \, t ~~\text{for all $t \in [0,h]$ and $y \in \mathbb{T}^{\m}_{\ell} \times \mathbb{R}^{3 \m}$} \;. 
 \] 
 For $s \ge 0$, the Markov property implies \begin{align*}
 &\E_y( M_t - M_s ) \, = \, e^{\cA \, s} \E_y( e^{\cA \, (t-s)} \rho_t - \rho_s ) \\
 & \qquad \, = \,  e^{\cA \, s} \E_y \Big( \E_{Y_s} \big( e^{\cA \,  (t-s)} \rho_{t-s} - \rho_0 \big) \Big) \, \le \,  e^{\cA \, s} \, \m \, C \, h \, (t-s) 
 \end{align*}
 for all $t \in [s, s+h]$ and $y \in \mathbb{T}^{\m}_{\ell} \times \mathbb{R}^{3 \m}$.  Hence, with $t_k := k h$ for all $k \in \mathbb{N}_0$, \[
 \E_y( M_t - M_0 ) \, = \, \sum_{k=1}^{\infty} \E_y \big( M_{t_k \wedge t} - M_{t_{k-1} \wedge t} \big) \, \le \, e^{\cA \, t} \,\m \, C \, h \, t
 \] for all $t \ge 0$ and $y \in \mathbb{T}^{\m}_{\ell} \times \mathbb{R}^{3 \m}$.  Letting $h \downarrow 0$ we obtain $\E_y ( M_t - M_0 ) \le 0 $, and thus, \begin{align*}
  & \E_y( M_t - M_s \mid \mathcal{F}_s ) \, = \, e^{\cA \, s} \E_y( e^{\cA \, (t-s)} \rho_t - \rho_s \mid \mathcal{F}_s ) \\
 & \qquad \, = \, e^{\cA \, s}\E_{Y_s}( e^{\cA \, (t-s)} \rho_{t-s} - \rho_0 ) = e^{\cA \, s} \E_{Y_s}( M_{t-s} - M_0 ) \, \le \, 0 \;.
 \end{align*} 
Hence, $M_t = e^{c_J \, t} \rho_t$ is a nonnegative supermartingale.

\medskip

Finally, we consider the process $Y_t$ with initial distribution given by an optimal coupling of the initial distributions $\nu$ and $\eta$ w.r.t.\ the distance $\mathcal W_{\rho}$, i.e., the law of $(x,v,\tilde x, \tilde v)$ has marginals $\mu$ and $ \nu $ and $Y_0 = (x,v,x-\tilde x,v-\tilde v)$, and  $\mathcal{W}_{\rho}(\mu,\nu) = \E[ \rho(Y_0)]$. Then,
for all $t\ge 0$,
the law of $Y_t$ represents a coupling of $\mu p_t$ and $ \nu p_t $, and hence by \eqref{eq:Wrho}, \[
\mathcal{W}_{\rho}(\mu p_t, \nu p_t) \ \le \ \E[ \rho(Y_t)] \ \le \ e^{-\cA \, t}  \E[ \rho(Y_0)]  \ \le \ e^{-\cA \, t}  \mathcal{W}_{\rho}(\mu , \nu) \;,
\]
which proves \eqref{eq::Wrho_tor}, as required.
\end{proof}


%
%
%

\section{Appendix} \label{appendix}
This Appendix briefly reviews and slightly adapts \cite[Theorem 5.5]{davis1984piecewise} to prove a supermartingale theorem for PDMPs needed in the proof of our main results.  To state this result, let $(\mathsf{X}_t)_{t \ge 0}$ be a PDMP with the following characteristics
\begin{enumerate}[label=(\roman*)]
\item 
boundaryless state space $S$;
\item
deterministic flow $\zeta_t: S \to S$ generated by a vector field $\mathfrak{X}: S \to \mathbb{R}^n$;
\item 
jump rates $\mathsf{J}(x)$ where $\mathsf{J}: S \to \mathbb{R}_{> 0}$; and,
\item
 jump measure $\mathsf{Q}(x,dy)$.  
 \end{enumerate}  
 On continuously differentiable functions $f$, define the generator of $\mathsf{X}_t$ as the operator $\mathcal{G}$ that outputs the function $\mathcal{G}f: S \to \mathbb{R}$ defined as \[
\mathcal{G} f(x) = ( \mathfrak{X} \cdot \nabla f) (x) + \mathsf{J}(x) \int_S \left( f(y) - f(x) \right)\mathsf{Q}(x, dy) \;.
\]  Let $g: [0, \infty) \times S \to \mathbb{R}$ be a space-time-dependent function.  For any $x \in S$, suppose that the function $t \mapsto g(t, \zeta_t(x))$ is absolutely continuous in time except at jump discontinuities where it is c\'adl\'ag and nonincreasing.  In this context, we prove that the process \[
g(t, \mathsf{X}_t) - \int_0^t \left( \frac{\partial g}{\partial t} +  \mathcal{G} g \right)(s, \mathsf{X}_s) ds
\] is a local supermartingale.

First, we recall that $\mathsf{X}_t$ solves the following time-dependent martingale problem \cite{davis1984piecewise, Da1993}.

\begin{lemma} \label{davis_l_martin}
For any $g: [0, \infty) \times S \to \mathbb{R}$ such that $g$ is differentiable in its first variable, $\mathbb{E} \sum_{s \le t} |g(s,\mathsf{X}_s) - g(s-,\mathsf{X}_{s-})| < \infty$ for each $t \ge 0$, and the function $t \mapsto g(t,\zeta_t(x))$ is absolutely continuous for all $x \in S$,  the process \[
g(t,\mathsf{X}_t) - \int_0^t \left( \frac{\partial g}{\partial t} + \mathcal{G} g \right) (s, \mathsf{X}_s) ds
\] is a local martingale.  
\end{lemma}

\begin{proof}
This is a special case of Theorem 5.5 of Ref.~\cite{davis1984piecewise} when the state space $S$ is boundaryless.
\end{proof}

Next we apply this result to functions of the process $\mathsf{X}_t$ that are piecewise absolutely continuous functions (in time) and 
that have nonincreasing jumps along the deterministic part of $\mathsf{X}_t$.

\begin{lemma} \label{davis_l_supermartin}
For any $g: [0, \infty) \times S \to \mathbb{R}$ such that $g$ is differentiable in its first variable; $\mathbb{E} \sum_{s \le t} |g(s,\mathsf{X}_s) - g(s-,\mathsf{X}_{s-})| < \infty$ for each $t \ge 0$ and for all $\mathsf{X}_0 \in S$; and the function $G(t): t \mapsto g(t,\zeta_t(x))$ is piecewise absolutely continuous, c\'adl\'ag, and $\Delta G(t) \le 0$ for all $t>0$ and for all $x \in S$; then the process \[
g(t,\mathsf{X}_t) - \int_0^t \left( \frac{\partial g}{\partial t} + \mathcal{G} g \right) (s, \mathsf{X}_s) ds
\] is a local supermartingale.  
\end{lemma}

\begin{proof}
The proof of this result is almost identical to the proof of Theorem 5.5 of \cite{davis1984piecewise} except that we must include the jumps in $g(t,\mathsf{X}_t)$ along the deterministic parts of $\mathsf{X}_t$.  Let $\{ t_i \}$ denote the jump times of the process $\mathsf{X}_t$. Then we have the following representation \begin{align*}
g(t,\mathsf{X}_t) - g(0,\mathsf{X}_0) &= \sum_{s \le t} \sum_{t_i \le t} \Delta g(s-t_i, \zeta_{s-t_i}(\mathsf{X}_{t_i})) 1_{s \in [t_i, t_{i+1})} \\
& + \int_0^t \left( \frac{\partial g}{\partial t} +  \mathcal{G} g \right)(s, \mathsf{X}_s) ds + M_t^g 
\end{align*} 
where $M_t^g$ is a local martingale.  Since the jumps in $g$ along the deterministic parts of $\mathsf{X}_t$ are nonincreasing everywhere,  \[
g(t,\mathsf{X}_t)  - \int_0^t \left( \frac{\partial g}{\partial t} +  \mathcal{G} g \right)(s, \mathsf{X}_s) ds \le g(0,\mathsf{X}_0)  + M_t^g \;.
\]  It follows that $g(t,\mathsf{X}_t) - \int_0^t \left( \frac{\partial g}{\partial t} +  \mathcal{G} g \right)(s, \mathsf{X}_s) ds$ is a local supermartingale.
\end{proof}

\begin{theorem} \label{davis_supermartin}
Suppose that $g: [0,\infty) \times S \to \mathbb{R}$ is nonnegative, satisfies the conditions of Lemma~\ref{davis_l_supermartin}, and satisfies $\int_0^t \left( \frac{\partial g}{\partial t} + \mathcal{G} g \right) (s, \mathsf{X}_s) ds \le 0$ for all $t \ge 0$ and for all $\mathsf{X}_0 \in S$.  Then $g(t,\mathsf{X}_t)$ is a supermartingale. 
\end{theorem}

\begin{proof}
Since the conditions for Lemma~\ref{davis_l_supermartin} hold, the process \[
g(t,\mathsf{X}_t) - \int_0^t \left( \frac{\partial g}{\partial t} + \mathcal{G} g \right) (s, \mathsf{X}_s) ds
\] is a local supermartingale.  Moreover, since $
 \int_0^t \left( \frac{\partial g}{\partial t} + \mathcal{G} g \right) (s, \mathsf{X}_s) ds \le 0$ by assumption, $g(t,\mathsf{X}_t)$ is also a local supermartingale. Since the function $g$ is also nonnegative by assumption, Fatou's lemma implies that $g(t,\mathsf{X}_t)$ is a supermartingale, as required.
\end{proof}

%
%
%
%
%
%

\section*{Acknowledgements}
N.~B-R. has been supported by the Alexander von Humboldt foundation and the National Science Foundation under Grant No.~DMS-1816378. 
 
A.~Eberle has been supported by the Hausdorff Center for Mathematics. Gef\"ordert durch die Deutsche Forschungsgemeinschaft (DFG) im Rahmen der Exzellenzstrategie des Bundes und der L\"ander - GZ 2047/1, Projekt-ID 390685813.

%
%
%


\bibliographystyle{amsplain}
\bibliography{nawaf}

\end{document}